\documentclass[reqno,11pt]{amsart}
\usepackage{amsmath, amsthm, amssymb, booktabs}
\usepackage{mathrsfs}
\usepackage{color}

\advance \topmargin by -\headheight
\advance \topmargin by -\headsep
     
\evensidemargin \oddsidemargin
\newtheorem{theorem}{Theorem}[section]
\newtheorem{lemma}[theorem]{Lemma}
\newtheorem{corollary}[theorem]{Corollary}
\newtheorem{conjecture}[theorem]{Conjecture}

\theoremstyle{definition}
\newtheorem{definition}[theorem]{Definition}

\theoremstyle{remark}

\numberwithin{equation}{section}

\newcommand{\mmod}[1]{\,\,(\text{mod}\,\,#1)}

\def\bfa{{\mathbf a}}
\def\bfb{{\mathbf b}}
\def\bfc{{\mathbf c}}
\def\bfd{{\mathbf d}}

 \def\bfP{{\mathbf P}}

\def\bfx{{\mathbf x}}
\def\bfy{{\mathbf y}}

\def\calA{{\mathscr A}}  

\def\calB{{\mathscr B}} 
\def\calC{{\mathscr C}} 

\def\calD{{\mathscr D}}
\def\calE{{\mathscr E}}
\def\calF{{\mathcal F}}

\def\calS{{\mathcal S}}

\def\dbC{{\mathbb C}}
\def\dbF{{\mathbb F}}
\def\dbN{{\mathbb N}}  
\def\dbR{{\mathbb R}}
\def\dbZ{{\mathbb Z}}\def\dbQ{{\mathbb Q}}

\def\gra{{\mathfrak a}}

\def\alp{{\alpha}} \def\bfalp{{\boldsymbol \alpha}}
\def\bet{{\beta}}  
\def\gam{{\gamma}} \def\Gam{{\Gamma}}

\def\del{{\delta}}

\def\tet{{\theta}}  
\def\Tet{{\Theta}} 
\def\kap{{\kappa}}
 \def\Lam{{\Lambda}}

\def\Ups{{\Upsilon}} 
 \def\bfvarphi{{\boldsymbol \varphi}}

\def\ome{{\omega}}

\def\eps{\varepsilon}

\def\le{\leqslant} \def\ge{\geqslant}

\def\d{{\,{\rm d}}}

\begin{document}
\title[Condensation and densification]{Condensation and densification for\\ 
sets of large diameter}
\author[Trevor D. Wooley]{Trevor D. Wooley}
\address{Department of Mathematics, Purdue University, 150 N. University Street, West 
Lafayette, IN 47907-2067, USA}
\email{twooley@purdue.edu}
\dedicatory{In memory of Ron Graham}
\date{}
\begin{abstract} Consider a set of integers $\calA$ having finite diameter $X$, and a system 
of simultaneous polynomial equations to be solved over $\calA$. In many circumstances, it is 
known that the number of solutions of this system is $O(X^\eps |\calA|^\tet)$ for a suitable 
$\tet>0$ and any $\eps>0$. These estimates become worse than trivial when the diameter 
$X$ is very large compared to $|\calA|$, or equivalently, when the set $\calA$ is very sparse. 
This motivates the problem of seeking a new set of integers $\calB$, in a certain sense 
isomorphic to $\calA$, having the property that the diameter $X'$ of $\calB$ is smaller than 
$X$, and at the same time the set $\calB$ preserves the salient features of the solution set 
of the system of equations in question. We report on our speculative investigations 
concerning this problem closely associated with the topic of Freiman homomorphisms.
\end{abstract}
\maketitle

\section{Introduction} Given a system of polynomial equations having integral coefficients, 
the investigation of solution sets with variables restricted to a given finite set of integers 
$\calA$ is of basic interest in arithmetic combinatorics. Even for a fixed system of 
equations, comprehensive knowledge concerning such solution sets seems a goal far too 
ambitious to be realised, for the sets $\calA$ to which variables are restricted may contain 
extraordinarily complicated constellations of arbitrarily large size. In this paper we seek to 
understand such solution sets in terms of related sets of integers, every element of which 
is bounded purely in terms of the cardinality of $\calA$ and the data associated with the 
system of polynomials in question. Thus, in a certain sense, our conclusions derive faithful 
models of solution sets in arithmetic combinatorics. Any model of this type having elements 
of least size might reasonably be interpreted as a minimal model. The interest in such 
models lies in the hope that a minimal model might be more easily understood than a 
non-minimal and potentially very sparse counterpart. There are close parallels with the 
concept of Freiman homomorphisms and isomorphisms (see \cite{Frei1964, Frei1973} and, 
for example, \cite[Definition 5.21]{TV2006}) in the situation wherein these mappings take 
one set of integers to another. Although we comment further on such considerations in due 
course, we emphasise for now the importance for us of remaining within the same ring 
rather than mapping to a finite field. We express the hope that, despite our investigations 
on these matters being rudimentary in nature, they may provide a stimulus for further work.

\par Further discussion requires the introduction of some notation. We are interested 
primarily in finite sets of integers $\calA\subset \dbZ$. We write $A$ for 
$\text{card}(\calA)$. Two notions of the size of the elements of $\calA$ play a role in our 
discussions. First, there is the {\it diameter} of $\calA$, namely
\[
\text{diam}(\calA)=\max \calA -\min \calA +1.
\]
Second, there is the {\it enveloping radius} of $\calA$, by which we mean
\[
\text{env}(\calA)=\max \{ |a|:a\in \calA\}+1.
\]
It is apparent that $\text{diam}(\calA)$ and $\text{env}(\calA)$ provide very crude 
measures of the complexity of the set $\calA$ in wide generality\footnote{The presence of 
the additional term $1$ in these definitions may seem mysterious, but is designed to align 
with a subsequent definition appropriate for the situation in algebraic number fields.}. One 
focus of interest for us concerns translation-dilation invariant (TDI) systems of equations, 
such as the familiar linear equation $x_1+x_2=x_3+x_4$. When considering the solutions of 
such an equation with $\bfx\in \calA^4$, it is apparent that no information concerning the 
solution set is lost if one translates the elements of $\calA$ by a fixed integer $b$ to obtain 
a new set $\calA'=\{ a-b:a\in \calA\}$. Consequently, there is no loss of generality in 
assuming that $\min \calA =0$, and in such circumstances it is more natural to measure the 
complexity of the set $\calA$ by means of its diameter rather than its enveloping radius.

\par The measures $\text{env}(\calA)$ and $\text{diam}(\calA)$ play a critical role in the 
best available upper bounds for certain mean values of additive number theory. For 
example, when $s$ and $k$ are natural numbers and $\calA\subset \dbZ$ is finite, let 
$J_{s,k}(\calA)$ denote the number of solutions of the system of equations
\[
x_1^j+\ldots +x_s^j=x_{s+1}^j+\ldots +x_{2s}^j\quad (1\le j\le k),
\]
with $x_i\in \calA$ $(1\le i\le 2s)$. Likewise, when $\varphi_j\in \dbZ[t]$ $(1\le j\le k)$, 
denote by $J_{s,k}(\calA;\bfvarphi)$ the number of solutions of the system of equations
\[
\varphi_j(x_1)+\ldots +\varphi_j(x_s)=\varphi_j(x_{s+1})+\ldots +\varphi_j(x_{2s})
\quad (1\le j\le k),
\]
with $x_i\in \calA$ $(1\le i\le 2s)$. We begin by recalling a consequence of recent work 
resolving the main conjecture in Vinogradov's mean value theorem.

\begin{theorem}\label{theorem1.1} Let $\calA\subset \dbZ$ be finite.
\begin{enumerate}
\item[(i)] Suppose that $\varphi_j\in \dbZ[t]$ $(1\le j\le k)$ is a system of polynomials with
\[
\text{\rm det}\biggl( \frac{{\rm d}^i\varphi_j(t)}{{\rm d}t^i}\biggr)_{1\le i,j\le k}\ne 0.
\]
Also, let $s$ and $k$ be natural numbers with $s\le k(k+1)/2$. Then for each $\eps>0$, 
one has
\begin{equation}\label{1.1}
J_{s,k}(\calA;\bfvarphi)\ll \text{\rm env}(\calA)^\eps A^s.
\end{equation}
\item[(ii)] For all natural numbers $s$ and $k$, and each $\eps>0$, one has
\begin{equation}\label{1.2}
J_{s,k}(\calA)\ll \text{\rm diam}(\calA)^\eps (A^s+A^{2s-k(k+1)/2}).
\end{equation}
\end{enumerate}
In each asymptotic bound, the constants implicit in Vinogradov's notation $\ll$ may depend 
on $\eps$, $s$, $k$ and the coefficients of $\bfvarphi$.
\end{theorem}

Both of the conclusions of Theorem \ref{theorem1.1} are immediate consequences of 
Wooley \cite[Theorem 1.1]{Woo2019}, as we explain in \S9 below, and the conclusion (ii) is 
also immediate from the work of Bourgain, Demeter and Guth \cite{BDG2016}. The 
motivating observation we wish to highlight here is that both estimates (\ref{1.1}) and 
(\ref{1.2}) are worse than trivial when the set $\calA$ is extremely sparse. Suppose, for 
example, that
\[
\text{diam}(\calA)\asymp \exp (\exp (A)).
\]
Then the estimates (\ref{1.1}) and (\ref{1.2}) are inferior to the trivial 
bounds $J_{s,k}(\calA;\bfvarphi)\le A^{2s}$ and $J_{s,k}(\calA)\le A^{2s}$. This 
observation remains valid for the improved estimates for $J_{3,2}(\calA)$ and 
$J_{6,3}(\calA)$ made available, respectively in the very recent work reported in 
\cite{GLY2021, GMW2020} and \cite{Sch2023}. It seems reasonable to speculate that the 
extremal situation is that in which $\calA$ consists of $A$ consecutive integers.

\begin{conjecture}\label{conjecture1.2} Suppose that $\calA\subset \dbZ$ is finite and 
$s,k\in \dbN$. Then
\[
J_{s,k}(\calA)\le J_{s,k}(\{1,2,\ldots ,A\}).
\]
Moreover, for each $\eps>0$, one has
\begin{equation}\label{1.4}
J_{s,k}(\calA)\ll_{\eps,s,k} A^{s+\eps}+A^{2s-k(k+1)/2}.
\end{equation}
\end{conjecture}

By elimination, one finds that the bound $J_{s,k}(\calA)\ll A^s+A^{2s-k}$ is essentially 
trivial. When $k\ge 2$, this estimate remains very far from that asserted in Conjecture 
\ref{conjecture1.2}. The only other non-trivial bound of which the author is aware is an 
estimate very slightly stronger than $J_{s,2}(\calA)\ll A^{2s-3+2^{2-s}}$ $(s\ge 3)$ due to 
Mudgal \cite[Theorem 1.1]{Mud2023}. With progress towards Conjecture 
\ref{conjecture1.2} in mind, it would be desirable to have available a set $\calC$ associated 
with a sparse set $\calA$ having the property that
\[
J_{s,k}(\calA)=J_{s,k}(\calC),
\]
or even merely
\[
J_{s,k}(\calA)\ll J_{s,k}(\calC),
\]
and, moreover, having much smaller diameter than $\calA$. Were one to have the upper 
bound $\text{diam}(\calC)\le A^c$, for some fixed $c>0$, for example, then the conjecture 
(\ref{1.4}) would follow from Theorem \ref{theorem1.1}(ii). Although such cannot be true 
in general, one is led to the broader problem of determining the extent to which such 
compressions might be achieved in practice. This problem concerning {\it condensations} is 
formalised in \S2, and explored in \S\S3, 4 and 5. We direct the reader to Theorem 
\ref{theorem5.6} for our most general conclusions concerning condensations associated with 
systems of polynomial equations. Write $c_1$ and $c_2$ for suitable positive constants. 
Then a very rough idea of these conclusions can be surmised if we note, first, that we are 
forced to work in a number field of degree as large as $\exp_2(c_1A)$, and second, that our 
condensations contain elements roughly of size $\exp_4(c_2A)$. Here, and throughout, we 
use $\exp_m(\cdot)$ to denote the $m$-fold iterated exponential function. Thus
\[
\exp_1(x)=e^x,\quad \exp_2(x)=e^{e^x},
\]
and so on.\par

An alternate strategy for obtaining bounds of the shape (\ref{1.4}) has a very different 
flavour. One might surmise that the difficulty in applying Theorem \ref{theorem1.1} to 
establish the estimate (\ref{1.4}) of Conjecture \ref{1.2} stems from the awkward nature 
of ultra-sparse sets $\calA$ having very large diameter compared to their cardinality $A$. 
One might therefore seek to obtain a much denser set $\calD$ associated with $\calA$, 
having the property that for some large integer $N$ one has
\[
\text{card}(\calD)=\left( \text{card}(\calA)\right)^N\quad \text{and}\quad J_{s,k}(\calA)
=J_{s,k}(\calD)^{1/N},
\]
while at the same time $\text{diam}(\calD)$ is not much larger than $\text{diam}(\calA)$. 
This set $\calD$ would be a much denser analogue of $\calA$ with the potential that 
$\text{diam}(\calD)\le \left( \text{card}(\calD)\right)^c$, for some fixed $c>0$. In these 
circumstances, the conjectured estimate (\ref{1.4}) would again follow from Theorem 
\ref{theorem1.1}(ii). We formalise this problem of {\it densification} in \S6 and explore it in 
\S7.\par

It seems worth remarking that the concepts of condensation and densification possess 
interpretations also in the scenario wherein sets of integers are replaced by finite sets of 
real numbers, or even finite subsets of a characteristic zero integral domain. We make 
some remarks in this direction in \S8.\par

We view both the strategies of condensation and densification of sets of large diameter as 
being of interest in their own right. We emphasise that our conclusions do not achieve the 
level whereby application to Conjecture \ref{conjecture1.2} can reasonably be envisioned.
\par

\noindent {\bf Acknowledgements:} This work was supported in its initial phases by a 
European Research Council Advanced Grant under the European Union's Horizon 2020 
research and innovation programme via grant agreement No.~695223. The bulk of the 
work reported here was obtained while the author was supported by the National Science 
Foundation via Grant No.~DMS-1854398 and DMS-2001549. The author wishes to express 
his gratitude to Julia Wolf for some early discussions on the topic of this paper, and to Ben 
Barber for discussion concerning an idea that led to Theorem \ref{theorem2.4}.\par

We write $X\asymp Y$ when, in Vinogradov's notation, we have $X\ll Y\ll X$. Also, when 
$\tet$ is a real number, we write $\lceil \tet \rceil$ for the least integer $m$ with 
$m\ge \tet$, and likewise $\lfloor \tet\rfloor$ for the largest integer $m$ with $m\le \tet$. 
In addition, we write $\|\tet\|$ for $\min \{ |\tet-m|:m\in \dbZ\}$. Finally, we make 
frequent use of vector notation in the form $\bfx=(x_1,\ldots,x_r)$. Here, the dimension 
$r$ depends on the course of the argument.\par
 
\section{Condensations of sets} The informal introduction of condensations in \S1 provides 
a framework insufficient for the more serious discussion on which we now embark. We 
begin by introducing a notion generalising that of a Freiman isomorphism.

\begin{definition}\label{definition2.1} Let $\calA$ and $\calB$ be finite sets of integers, 
and suppose that we are given polynomials $P_1,\ldots ,P_r\in \dbZ[x_1,\ldots ,x_s]$. We 
say that a bijection $\psi:\calA\rightarrow \calB$ is a {\it Freiman $\bfP$-isomorphism} if, 
whenever $(x_1,\ldots ,x_s)\in \calA^s$, then
\[
P_i(x_1,\ldots ,x_s)=0\quad (1\le i\le r)
\]
if and only if
\[
P_i(\psi(x_1),\ldots ,\psi(x_s))=0\quad (1\le i\le r).
\]
\end{definition}

We emphasise here that a Freiman $\bfP$-isomorphism is specific to a particular polynomial 
tuple $\bfP$, since our focus will lie on the solution set of a fixed polynomial system. This is 
in contrast with a similar definition given in work of Grosu (see the preamble to the 
statement of \cite[Theorem 1.3]{Gro2014}). Moreover, also in contrast to the latter and 
indeed other sources concerning Freiman isomorphisms, we shall only be interested in 
situations wherein both $\calA$ and $\calB$ lie in the same ring. This restriction permits 
iterative approaches in which one composes a sequence of Freiman $\bfP$-isomorphisms 
$\psi_1,\psi_2,\ldots ,\psi_n$ to obtain a new Freiman $\bfP$-isomorphism $\psi_n\circ 
\psi_{n-1}\circ \ldots \circ \psi_1$.\par

A few words of explanation seem warranted concerning our interest in Freiman 
$\bfP$-isomorphisms. We are interested in the structure of the solutions of the system of 
polynomials
\begin{equation}\label{2.1}
P_i(x_1,\ldots ,x_s)=0\quad (1\le i\le r),
\end{equation}
with $\bfx\in \calA^s$. This is described precisely by the hypergraph $\Gam(\calA;\bfP)$ 
with the elements of $\calA$ as vertices, and having hyperedges defined by the $s$-tuples 
$\bfx$ from $\calA^s$ satisfying the system of equations (\ref{2.1}). With this 
characterisation of the structure of the solution set of (\ref{2.1}) in mind, it is apparent 
that the mapping
\[
\Psi: \Gam(\calA;\bfP)\rightarrow \Gam(\calB;\bfP),
\]
induced by a Freiman $\bfP$-isomorphism $\psi:\calA \rightarrow \calB$, delivers a bijection 
that preserves every feature of the solution set of (\ref{2.1}) as one replaces $\calA$ by 
$\calB=\psi(\calA)$.\par

Given a finite set of integers $\calA$ and a system of polynomials $\bfP\in \dbZ[\bfx]^r$, 
our interest lies in finding a set $\calB$ Freiman $\bfP$-isomorphic to $\calA$ with $\calB$ 
having elements intrinsically smaller than those of $\calA$. Since $\Gam(\calB;\bfP)$ is in 
bijective correspondence with $\Gam(\calA;\bfP)$, one may expect that the salient features 
of the solution structure of the system (\ref{2.1}) with $\bfx\in \calA^s$ may be more easily 
determined by instead considering solutions $\bfx\in \calB^s$. This motivates the next 
definition.

\begin{definition}\label{definition2.2} We say that a mapping $\psi:\calA\rightarrow \calC$ 
is a {\it $\bfP$-condenser of $\calA$} if it is a Freiman $\bfP$-isomorphism having the 
property that $\text{env}(\calC)\le \text{env}(\calA)$. When the latter inequality is strict, 
we refer to $\psi$ as a {\it strict $\bfP$-condenser of $\calA$}. In either case, we refer to 
$\calC$ as being a {\it $\bfP$-condensation of $\calA$}.
\end{definition}

We make an observation here concerning TDI systems of polynomials $\bfP$. Suppose that 
$\min \calA =b$. Then by considering the mapping $\psi:\calA\rightarrow \dbZ$ defined by 
$a\mapsto a-b$, we see that $\calA$ possesses a $\bfP$-condensation $\calB$ with 
$\text{env}(\calB)=\text{diam}(\calB)$.\par

Of particular interest are the $\bfP$-condensations $\calB$ of $\calA$ distinguished by the 
property that $\text{env}(\calB)$ is minimal.

\begin{definition}\label{definition2.3} The {\it $\bfP$-essential enveloping radius} of 
$\calA$ is
\[
\text{env}^*(\calA;\bfP)=\min \{ \text{env}(\psi(\calA)):\text{$\psi$ is a $\bfP$-condenser 
of $\calA$}\},
\]
and the {\it $\bfP$-essential diameter of $\calA$} is
\[
\text{diam}^*(\calA;\bfP)=\min \{ \text{diam}(\psi(\calA)):\text{$\psi$ is a 
$\bfP$-condenser of $\calA$}\} .
\]
\end{definition}

The notion of the $\bfP$-essential enveloping radius of a finite set $\calA\subset \dbZ$ 
provides a measure of the complexity of $\calA$ with respect to the system of equations 
(\ref{2.1}), for it describes the minimal footprint of a set $\calB$ for which the hypergraph 
$\Gam(\calB;\bfP)$ associated with the solution set faithfully describes that of interest, 
namely $\Gam(\calA;\bfP)$.\par

In general, the sharpest conclusions concerning $\text{env}^*(\calA;\bfP)$ of which the 
author is aware are the trivial ones recorded in the following theorem.

\begin{theorem}\label{theorem2.4} Let $\bfP\in \dbZ[x_1,\ldots ,x_s]^r$ be a polynomial 
system, and let $\calA\subset \dbZ$ be a finite set of integers having cardinality $A$. Then 
one has
\[
\lceil (A+1)/2\rceil \le \text{\rm env}^*(\calA;\bfP)\ll_{A,\bfP} 1
\]
and
\[
A\le \text{\rm diam}^*(\calA;\bfP)\ll_{A,\bfP} 1.
\]
\end{theorem} 

We emphasise here that the upper bounds recorded in this theorem indicate that the 
$\bfP$-essential enveloping radius (respectively, the $\bfP$-essential diameter) of $\calA$ 
depends at most on $A=\text{card}(\calA)$ and the polynomials comprising $\bfP$, but not 
on the specific identity of the elements of $\calA$. A few moments of reflection should 
disabuse the puzzled reader that this conclusion might be in any sense non-trivial.

\begin{proof}[The proof of Theorem \ref{theorem2.4}] The solution set of the polynomial 
system 
\[
P_i(x_1,\ldots ,x_s)=0\quad (1\le i\le r),
\]
with $\bfx\in \calA^s$, defines the hypergraph $\Gam(\calA;\bfP)$. Let $\calC\subset \dbZ$ 
be any set of integers of cardinality $A$ having smallest enveloping radius for which 
$\Gam(\calC;\bfP)$ is isomorphic to $\Gam(\calA;\bfP)$. Denote by $\psi$ the mapping 
from $\calA$ to $\calC$ induced by this hypergraph isomorphism, and note that one 
possibility is that $\psi$ is the identity mapping. The definitions of $\Gam(\calC;\bfP)$ and 
$\Gam(\calA;\bfP)$ ensure that $\psi:\calA\rightarrow \calC$ is a bijection satisfying the 
property that whenever $(x_1,\ldots ,x_s)\in \calA^s$, then
\[
P_i(x_1,\ldots ,x_s)=0\quad (1\le i\le r)
\]
if and only if
\[ 
P_i(\psi(x_1),\ldots ,\psi(x_s))=0\quad (1\le i\le r).
\]
Hence, we see that $\psi$ is a Freiman $\bfP$-isomorphism and also a $\bfP$-condenser 
of $\calA$ with $\calC=\psi(\calA)$.\par

The set of all hypergraphs on $A$ vertices with hyperedges defined by $s$-tuples of 
vertices is finite in number. Indeed, the number of such hypergraphs depends at most on 
$s$ and $A$. Thus, since $\calC$ depends at most on the hypergraph isomorphism class 
of $\Gam(\calA;\bfP)$ and the polynomial system $\bfP$, one sees that $\text{env}(\calC)$ 
depends at most on $s$, $\bfP$ and $A$. Since $\calC=\psi(\calA)$ with $\psi$ a 
$\bfP$-condenser of $\calA$, it follows that $\text{env}^*(\calA;\bfP)\ll_{A,\bfP}1$. A 
similar conclusion is apparent also for $\text{diam}^*(\calA;\bfP)$ by arguing mutatis 
mutandis.\par

The lower bounds $\text{env}^*(\calA;\bfP)\ge \lceil (A+1)/2\rceil$ and 
$\text{diam}^*(\calA;\bfP)\ge A$ follow by considering sets $\calA$ containing $A$ 
consecutive integers.
\end{proof}

\section{Condensations for linear systems of equations} There is one class of polynomial 
systems for which the quantitative aspects of condensations are explicit, and for which the 
underlying methods possess familiar themes. Thus, the analysis of systems of linear 
polynomials $\bfP(\bfx)$ is both simple and instructive, and serves as a warm-up for the 
analysis of the next two sections concerning polynomial systems. We focus in this section on 
such linear systems in order to motivate the more general discussion of the next section.

\par In order to fix ideas, suppose that $s\ge 2$, $r\ge 1$, and for $1\le i\le r$ fix 
$b_i\in \dbZ$ and $c_{ij}\in \dbZ$ $(1\le j\le s)$. We ignore the trivial situation in which for 
some index $i$ one has $c_{ij}=0$ for $1\le j\le s$, since this will correspond to a case in 
which $r$ is smaller. The system of polynomials of interest to us in this section is
\[
P_i(\bfx)=\sum_{j=1}^sc_{ij}x_j-b_i\quad (1\le i\le r).
\]
When $\calA\subset \dbZ$ is a finite set of integers, we write $S(\calA;\bfP)$ for the set of 
solutions of the system of equations $P_i(\bfx)=0$ $(1\le i\le r)$, with $\bfx\in \calA^s$. 
Also, we define the integer $\Lam=\Lam (\bfc,\bfb)$ by putting
\[
\Lam =\max_{1\le i\le r}\biggl( |b_i|+\sum_{j=1}^s|c_{ij}|\biggr).
\]
Thus, the quantity $\Lam$ provides a measure of the height of the coefficient matrix 
defining $\bfP$. Finally, it is convenient both here and elsewhere to introduce an integer 
that encapsulates both distinctness of the elements of $\calA$, and also whether or not an 
$s$-tuple $\bfa$ lies in $S(\calA;\bfP)$. Thus, we define
\begin{equation}\label{3.1}
\Ups=\Biggl( \prod_{\substack{a_1,a_2\in \calA\\ a_1\ne a_2}}|a_1-a_2|\Biggr) 
\Biggl( \prod_{\substack{\bfa\in \calA^s\\ \bfa\not\in S(\calA;\bfP)}}
\sum_{i=1}^r|P_i(\bfa)|\Biggr) .
\end{equation}

\begin{theorem}\label{theorem3.1} Consider a system $\bfP$ of linear polynomials as 
described in the preamble, and consider a finite set of integers $\calA$. Then provided that 
$A=\text{\rm card}(\calA)$ is sufficiently large in terms of $r$ and $s$, one has
\begin{equation}\label{3.2}
\text{\rm env}^*(\calA;\bfP)<\exp\left(3(\Lam+1)^A\right) .
\end{equation}
\end{theorem}

\begin{proof} A moment of reflection reveals that there is no loss of generality in 
supposing that $\Lam\ge 2$ and $s\ge 2$. Write $X=\text{env}(\calA)-1$. Our strategy is 
to find an integer $h$ with $\Lam<h<2X$ having the following three properties:
\begin{enumerate}
\item[(i)] when $a_1,a_2\in \calA$ satisfy $a_1\ne a_2$, then 
$a_1\not\equiv a_2\mmod{h}$;
\item[(ii)] when $\bfa\not \in S(\calA;\bfP)$, then there is an index $i$ with $1\le i\le r$ for 
which one has $P_i(\bfa)\not\equiv 0\mmod{h}$;
\item[(iii)] for every $a\in \calA$, one has $\|a/h\|<1/\Lam$.
\end{enumerate}
If such an integer $h$ can be found, then we may define the map $\psi:\calA\rightarrow 
\dbZ$ by putting
\[
\psi(a)=\left[ a\mmod{h}\right],
\]
where $\left[ a\mmod{h}\right]$ denotes the numerically least residue of $a$ modulo 
$h$. To be clear, the numerically least residue of $a$ modulo $h$ is the integer $m$ with 
$-h/2<m\le h/2$ for which $a\equiv m\mmod{h}$. Property (i) then ensures that the set 
$\calB=\psi(\calA)$ is in bijective correspondence with $\calA$. Also, property (ii) ensures 
that whenever $\bfa\not\in S(\calA;\bfP)$, then for some index $i$ with $1\le i\le r$, one 
has
\[
P_i(\psi(\bfa))\equiv P_i(\bfa)\not\equiv 0\mmod{h},
\]
whence $P_i(\psi(\bfa))\neq 0$. However, when $\bfa\in \calS(\calA;\bfP)$, one has
\[
P_i(\psi(\bfa))\equiv P_i(\bfa)\equiv 0\mmod{h}\quad (1\le i\le r).
\]
At the same time, in view of property (iii), one has
\[
|P_i(\psi(\bfa))|<\biggl(|b_i|+\sum_{j=1}^s|c_{ij}|\biggr)\frac{h}{\Lam}\le h,
\]
whence $P_i(\psi(\bfa))=0$ $(1\le i\le r)$. We therefore infer that the map $\psi$ is a 
Freiman $\bfP$-isomorphism from $\calA$ to $\calB$. Consequently, since
\[
\text{env}(\calB)-1\le h/2<X=\text{env}(\calA)-1,
\]
we have confirmed the existence of a strict $\bfP$-condenser of $\calA$.\par

We establish the existence of a suitable integer $h$ by modifying very slightly an argument 
employed by Baker and Harman (see \cite[Proposition 1]{BH1996}). Recall the definition 
(\ref{3.1}) of the integer $\Ups$. A crude estimate delivers the bounds
\begin{equation}\label{3.3}
1\le \Ups\le (2X)^{A^2-A}(r\Lam X)^{A^s}\le \tfrac{1}{3}(r\Lam X)^{2A^s}.
\end{equation}
The number of prime divisors of $\Ups$ exceeding $\log (3\Ups)$ cannot exceed
\[
\frac{\log (3\Ups)}{\log \log (3\Ups)}.
\]
Thus, an application of the prime number theorem reveals that whenever $Y$ is large and 
$Y\ge 2\log (3\Ups)$, then in any interval $(Y,2Y)$, there exists a prime $\pi$ with 
$\pi\nmid \Ups$. It therefore follows from (\ref{3.3}) that we may choose a prime $\pi$ 
with $\pi\nmid \Ups$ for which
\begin{equation}\label{3.4}
\pi<4\log (3\Ups)\le 8A^s\log (r\Lam X).
\end{equation}
We put
\[
L=\text{lcm}[1,2,\ldots ,(\Lam+1)^A]
\]
and note that an elementary application of the prime number theorem ensures that, when 
$A$ is large, one has $L\le \exp(2(\Lam+1)^A)$. Next, by applying the multidimensional 
version of Dirichlet's box principle to the real numbers $a/(\pi L)$ $(a\in \calA)$, it follows 
that for some $\rho\in \dbN$ with 
$1\le \rho\le (\Lam+1)^A$, one has
\[
\left\| \frac{\rho a}{\pi L}\right\| \le \frac{1}{\Lam+1}\quad (a\in \calA).
\]
Since $\rho|L$, we may therefore define the integer $h=\pi L/\rho$, and then we see that
\begin{equation}\label{3.5}
\| a/h\|\le (\Lam+1)^{-1}\quad (a\in \calA).
\end{equation}

\par By construction, the integer $h$ is divisible by $\pi$. We now exploit the fact that 
$\pi\nmid \Ups$ using the definition \eqref{3.1}. Thus, when $a_1,a_2\in \calA$ satisfy 
$a_1\ne a_2$, one has $a_1\not\equiv a_2\mmod{\pi}$. Moreover, when $\bfa\in \calA^s$ 
one sees that $P_i(\bfa)\equiv 0\mmod{\pi}$ $(1\le i\le r)$ if and only if 
$\bfa\in S(\calA;\bfP)$. In combination with \eqref{3.5}, therefore, it is apparent that the 
properties (i), (ii) and (iii) above all hold for the integer $h$ that we have constructed. In 
particular, the map $\psi:\calA \rightarrow \dbZ$ defined by putting $\psi(a)=[a\mmod{h}]$ 
gives a Freiman $\bfP$-isomorphism from $\calA$ to $\calB=\psi(\calA)$ in which, on 
recalling (\ref{3.4}), we see that
\[
\text{env}(\calB)-1\le h/2\le \pi L/2<4A^s\log (r\Lam X)\text{exp}(2(\Lam+1)^A).
\]

\par We may summarise our deliberations thus far in the following form. Whenever $\calA$ 
is a finite subset of $\dbZ$ with cardinality $A$ and enveloping radius $X+1$, then $\calA$ 
possesses a $\bfP$-condensation $\calA_1=\psi(\calA)$ with enveloping radius at most 
$X_1+1$, where
\[
X_1=4A^s\log (r\Lam X)\text{exp}\left(2(\Lam+1)^A\right).
\]
When $A$ is large in terms of $r$ and $s$, and $X+1\ge \exp(3(\Lam+1)^A)$, we have
\[
4A^s\left( \log (r\Lam )+\log X\right) \le \tfrac{1}{2}X^{1/3}.
\]
Thus, under the same conditions on $A$ and $X$, it follows that
\[
X_1\le \tfrac{1}{2}X^{1/3}(X+1)^{2/3}<X,
\] 
and consequently one has $\text{env}(\calA_1)<\text{env}(\calA)$. Provided that 
\[
\text{env}(\calA_1)\ge \exp\left( 3(\Lam+1)^A\right),
\]
we may apply this process again, next showing that $\calA_1$ has a $\bfP$-condensation 
$\calA_2$ with $\text{env}(\calA_2)<\text{env}(\calA_1)$. Since $\bfP$-condensers may be 
composed, it follows that $\calA$ also has a $\bfP$-condensation $\calA_2$ with enveloping 
radius smaller than $\text{env}(\calA_1)$. By iterating this process repeatedly, with each 
iteration reducing the enveloping radius of the condensation of $\calA$, we ultimately obtain 
a condensation $\calA^*$ of $\calA$ for which
\[
\text{env}(\calA^*)<\exp\left( 3(\Lam+1)^A\right) .
\]
This establishes the bound \eqref{3.2}, and the proof of the theorem is complete.
\end{proof}

In the situation in which $\bfb={\mathbf 0}$, so that all of the linear polynomials are 
homogeneous, there is more freedom to apply changes of variable to advantage. Here the 
arguments are reminiscent of those employed in the proof of versions of Freiman's 
theorem (see, for example, the proof of \cite[Theorem 2.1]{BLR1998}).

\begin{theorem}\label{theorem3.2} Consider a system $\bfP$ of linear polynomials with 
$\bfb={\mathbf 0}$, as described in the preamble to the statement of Theorem 
\ref{theorem3.1}. Also, consider a finite set of integers $\calA$. Then provided that 
$A=\text{\rm card}(\calA)$ is sufficiently large in terms of $r$ and $s$, one has
\begin{equation}\label{3.6}
\text{\rm env}^*(\calA;\bfP)\le (\Lam+1)^A.
\end{equation}
\end{theorem}

\begin{proof} We proceed much as in the proof of Theorem \ref{theorem3.1}, though 
with a twist en route. First, writing $X=\text{env}(\calA)-1$ and defining the integer $\Ups$ 
as in (\ref{3.1}), we again obtain the bound (\ref{3.3}), and conclude that there exists a 
prime number $\pi$ with $\pi\nmid \Ups$ satisfying the property that
\begin{equation}\label{3.7}
(\Lam+1)^A<\pi<2\max\{ (\Lam+1)^A,4A^s\log (r\Lam X)\}.
\end{equation}
By applying the multidimensional version of Dirichlet's approximation theorem to the real 
numbers $a/\pi$ $(a\in \calA)$, it follows that for some $\rho\in \dbN$ with 
$1\le \rho\le (\Lam+1)^A$, one has
\[
\left\| \frac{\rho a}{\pi}\right\|\le \frac{1}{\Lam+1}\quad (a\in \calA).
\]
We fix any such integer $\rho$, noting that since $\pi>\rho$, one has $(\rho,\pi)=1$. It 
follows that:
\begin{enumerate}
\item[(i)] whenever $a_1,a_2\in \calA$ satisfy $a_1\ne a_2$, then 
$\rho a_1\not\equiv \rho a_2\mmod{\pi}$;
\item[(ii)] whenever $\bfa\not\in S(\calA;\bfP)$, then there is an index $i$ with $1\le i\le r$ 
for which one has $P_i(\rho \bfa)\not\equiv 0\mmod{\pi}$;
\item[(iii)] for every $a\in \calA$, one has $\| \rho a/\pi\|<1/\Lam$.
\end{enumerate} 
\par

We now define the map $\psi:\calA\rightarrow \dbZ$ by putting
\[
\psi(a)=[\rho a\mmod{\pi}].
\]
Property (i) then ensures that the set $\calC=\psi(\calA)$ is in bijective correspondence 
with $\calA$. Also, property (ii) ensures that whenever $\bfa\not\in S(\calA;\bfP)$, then for 
some index $i$ with $1\le i\le r$, one has
\[
P_i(\psi(\bfa))\equiv \rho P_i(\bfa)\not\equiv 0\mmod{\pi},
\]
whence $P_i(\psi(\bfa))\neq 0$. However, when $\bfa\in S(\calA;\bfP)$, one has
\[
P_i(\psi(\bfa))\equiv \rho P_i(\bfa)\equiv 0\mmod{\pi}\quad (1\le i\le r).
\]
At the same time, in view of property (iii), one has
\[
|P_i(\psi(\bfa))|<\frac{\pi}{\Lam}\sum_{j=1}^s|c_{ij}|\le \pi,
\]
whence $P_i(\psi(\bfa))=0$ $(1\le i\le r)$. We therefore infer that $\psi$ is a Freiman 
$\bfP$-isomorphism from $\calA$ to $\calC$. Consequently, provided that $\pi<2X$, we see 
that
\[
\text{env}(\calC)-1\le \pi/2<X=\text{env}(\calA)-1,
\]
and thus we have established the existence of a strict $\bfP$-condenser of $\calA$.\par

Notice here that in view of (\ref{3.7}), one has
\[
\text{env}(\calC)-1\le \pi/2<\max\{ (\Lam+1)^A,4A^s\log (r\Lam X)\},
\]
and we again have available an iterative process for reducing the enveloping radius of 
$\bfP$-condensations of $\calA$ similar to that made available in the concluding stages of 
the proof of Theorem \ref{theorem3.1}. When $A$ is sufficiently large in terms of $r$ and 
$s$, and $X\ge (\Lam+1)^A$, we have
\[
4A^s(\log (r\Lam)+\log X)<X.
\]
In this instance, therefore, under the same conditions on $A$ and $X$, we discern from 
\eqref{3.7} that $\pi<2X$ and hence that $\text{env}(\calC)<\text{env}(\calA)$. Thus, our 
iteration continues until we obtain a $\bfP$-condensation $\calA^*$ of $\calA$ for which 
$\text{env}(\calA^*)-1<(\Lam+1)^A$. This confirms the bound \eqref{3.6}, and thus the 
proof of the theorem is complete.
\end{proof}

The problem of obtaining lower bounds on $\text{env}^*(\calA;\bfP)$ has been considered 
in special cases such as that in which the system $\bfP$ consists of the single polynomial 
$x_1+x_2=x_3+x_4$. Here, it is apparent that the set $\calA=\{0,1,2,4,\ldots ,2^{A-2}\}$ 
cannot be condensed into a fundamentally smaller set (see \cite[\S5]{KL2000}). Thus, in this 
special case, one has $\text{env}^*(\calA;\bfP)\gg 2^A$, and it is apparent that the upper 
bound on $\text{env}^*(\calA;\bfP)$ provided by Theorem \ref{theorem3.2} cannot in 
general be replaced by a quantity subexponential in $A$. 

\section{Condensations for non-linear systems of equations, I}
Equipped with the discussion of \S3 applicable to linear equations, we move on in this 
section to consider the corresponding situation for the solubility of polynomial equations of 
degree exceeding one over a finite subset $\calA$ of the integers. Here, any attempt to 
merely mimic the proofs of Theorems \ref{theorem3.1} and \ref{theorem3.2} must plainly 
be abandoned. Suppose, for example, that we seek to analyse the solubility with 
$\bfx\in \calA^4$ of the equation $x_1^2+x_2^2=x_3^2+x_4^2$ by utilising the map 
$\psi:\dbZ \rightarrow \dbZ/h\dbZ$ defined by putting $\psi(a)=[a\mmod{h}]$ for a suitable 
positive integer $h$. The optimistic notion that the congruence
\[
\psi(a_1)^2+\psi(a_2)^2\equiv \psi(a_3)^2+\psi(a_4)^2\mmod{h}
\]
might imply that
\[
\psi(a_1)^2+\psi(a_2)^2=\psi(a_3)^2+\psi(a_4)^2
\]
would seem to demand a choice for $h$ ensuring that $|\psi(a)|<\tfrac{1}{4}h^{1/2}$ 
for all $a\in \calA$. Such an eventuality cannot reasonably be countenanced for any but 
very special sets $\calA$. However, a means of mapping subsets of finite fields into subsets 
of $\dbC$, while preserving associated solution structures, has been made available in work 
of Grosu \cite{Gro2014}. With care, this approach can be wrought to yield a 
$\bfP$-condenser of sorts in the non-linear situation currently of interest to us.\par

The discussion of this section and the next requires the introduction of notions somewhat 
more flexible than those defined in \S2. We have in mind now that the sets of integers 
under consideration will be replaced by elements of some algebraic number field. For the 
sake of simplicity, we shall restrict the polynomial equations under consideration to have 
coefficients lying in $\dbZ$, though it is straightforward to relax this condition so that the 
coefficient ring $\dbZ$ is replaced by the ring of integers from some other number field.\par

\begin{definition}\label{definition4.1}
Let $\calA$ and $\calB$ be finite sets of algebraic numbers. Suppose that 
$P_1,\ldots ,P_r\in \dbZ[x_1,\ldots ,x_s]$. We say that a bijection 
$\psi:\calA\rightarrow \calB$ is an {\it algebraic Freiman $\bfP$-isomorphism} if, whenever 
$(x_1\ldots ,x_s)\in \calA^s$, then
\[
P_i(x_1,\ldots ,x_s)=0\quad (1\le i\le r)
\]
if and only if
\[
P_i(\psi(x_1),\ldots ,\psi(x_s))=0\quad (1\le i\le r).
\]
\end{definition}

Notice here that we have not insisted that $\calA$ and $\calB$ lie in the same number field. 
Thus, for example, one could have $\calA\subset \dbZ[\sqrt{-1}]$ and 
$\calB\subset \dbZ[\sqrt[3]{2}]$. Given this flexibility in the choice of the image set, an 
appropriate definition of the analogue of a $\bfP$-condenser takes some care. First, when 
$\calC=\{c_1,\ldots ,c_n\}$ is a finite set of algebraic numbers, we define the number field 
$K=\dbQ(\calC)$ by putting $K=\dbQ(c_1,\ldots ,c_n)$. We then put
\[
d(\calC)=[K:\dbQ]\quad \text{and}\quad D(\calC)=\text{Disc}(K:\dbQ).
\]
Rather than become entangled with a coordinate basis for $K$ over $\dbQ$, we instead 
work with minimal polynomials associated with each element $c\in \calC$. Here, by the 
minimal polynomial $m_c\in \dbZ[x]$ of $c\in \calC$, we mean the irreducible polynomial in 
$\dbZ[x]$ with content $1$ and positive leading coefficient satisfying the condition that 
$m_c(c)=0$. Note that if $m_c$ has leading coefficient $l$, then $l^{-1}m_c$ is the 
conventional minimal polynomial of $c$ over $\dbQ$. Given a polynomial $f\in \dbZ[x]$ with 
$f(x)=f_0+f_1x+\ldots +f_dx^d$, we define
\[
\|f\|_q=\left( |f_0|^q+|f_1|^q+\ldots +|f_d|^q\right)^{1/q}\quad (q=1,2).
\]
Then, as a measure of the enveloping radius of the set $\calC$, we work with the 
{\it algebraic enveloping radius} 
\[
\text{Env}(\calC)=\max\{ \| m_c\|_1 : c\in \calC\}.
\]
If $\calC$ is a set of rational integers, then it is apparent that 
$\text{Env}(\calC)=\text{env}(\calC)$. Notice that $\text{Env}(\calC)$ is independent of any 
particular coordinate basis for $\dbQ(\calC)$.\par

Our goal is now to map a set of algebraic numbers $\calA$, having a large algebraic 
enveloping radius $\text{Env}(\calA)$, to a new set $\calB$ having smaller algebraic 
enveloping radius $\text{Env}(\calB)$, via an algebraic Freiman $\bfP$-isomorphism 
$\psi:\calA\rightarrow \calB$. In this way, the size of the elements of $\calB$ is morally 
speaking smaller than the corresponding size of the elements of $\calA$, and yet $\calB$ 
preserves the salient features of the solubility of the system $\bfP(\bfx)={\mathbf 0}$ 
exhibited by $\calA$. This objective motivates the following analogues of Definitions 
\ref{definition2.2} and \ref{definition2.3}.

\begin{definition}\label{definition4.2}
We say that a mapping $\psi:\calA\rightarrow \calC$ is a 
{\it $d$-algebraic $\bfP$-condenser} of $\calA$ if it is an algebraic Freiman 
$\bfP$-isomorphism having the property that
\[
\text{Env}(\calC)\le \text{Env}(\calA)\quad \text{and}\quad d(\calC)=d.
\]
When the inequality here is strict, we refer to $\psi$ as a {\it strict} $d$-algebraic 
$\bfP$-condenser of $\calA$. In either case, we refer to $\calC$ as being a 
{\it $d$-algebraic condensation} of $\calA$.
\end{definition}

\begin{definition}\label{definition4.3}
Let $\calA$ be a finite set of algebraic numbers, and denote by $\Psi_\bfP(\del)$ the set of 
all $d$-algebraic $\bfP$-condensers of $\calA$ with $d\le \del$. Then the {\it $\del$-limited 
$\bfP$-essential enveloping radius} of a finite set $\calA$ of algebraic integers is
\[
\text{Env}_\del^*(\calA;\bfP)=\min \{ \text{Env}(\psi(\calA)):\psi\in \Psi_\bfP(\del) \}.
\]
\end{definition}

We are now equipped to describe, in broad and rough terms, our strategy for condensing 
algebraic sets $\calA$ into sets $\calC$ establishing that $\text{Env}_\del^*(\calA;\bfP)$ is 
bounded purely in terms of $|\calA|$ and $\bfP$, while at the same time maintaining $\del$ 
to be likewise bounded purely in terms of $|\calA|$ and $\bfP$. The details of this process 
will be the subject of the next section.\par

Let $\calA$ be a finite set of algebraic integers with $\text{card}(\calA)=A$. In fact we shall 
need to consider finite sets of algebraic numbers, and this generates additional complications 
relative to the situation where the algebraic numbers are in fact algebraic integers. 
However, this simplified case allows us to sketch out the necessary argument. Put 
$X=\text{Env}(\calA)$ and suppose that $d(\calA)=d_0$, with $d_0$ bounded above by 
some absolute constant.\par

Our first step is to seek a rational prime number $\pi$ having the property that 
$\pi \nmid (a_1-a_2)$ for any distinct elements $a_1,a_2\in \calA$, and also 
$\pi\nmid P_i(\bfa)$ for any $\bfa\in \calA^s$ with $P_i(\bfa)\ne 0$ $(1\le i\le r)$. For the 
sake of concreteness, we shall in fact interpret these divisibility relations by taking norms 
of the algebraic numbers in question. It is apparent that an argument similar to that applied 
in \S3 will deliver such a prime to us with $\pi\ll_{A,\bfP}\log X$. Unfortunately, we must 
ensure that this prime number behaves congenially with respect to the number field 
$K_0=\dbQ(\calA)$, because we intend subsequently to consider the set $\calA$ modulo 
$\pi$ to be a subset of the finite field $\dbF_\pi$, and thence consider the associated 
system of congruences
\[
P_i(\bfa)\equiv 0\mmod{\pi}\quad (1\le i\le r).
\]
We therefore seek an appropriately sized prime $\pi$ having the property that a certain 
minimal polynomial associated with $\calA$ splits into linear factors over $\dbF_\pi$. If we 
assume a certain Generalised Riemann Hypothesis, then it follows from an appropriate 
version of the Chebotarev density theorem that such a prime can be shown to exist with 
$\pi \ll_{A,\bfP}(\log X)^3$.\par

Having obtained a prime $\pi$ with the properties just described, our second step is to 
apply the argument of Grosu \cite{Gro2014} to rectify the set $\calA\mmod{\pi}$. Provided 
that $\pi$ is chosen large enough in terms of $A$ and $\bfP$, this argument shows that the 
set $\calA\mmod{\pi}$ can be mapped to a new set $\calB$ algebraic Freiman 
$\bfP$-isomorphic to $\calA$, and having the property that
\[
\text{Env}(\calB)\ll_{A,\bfP}(\log X)^{c(A,\bfP)}\quad \text{and}\quad 
d(\calB)\ll_{A,\bfP}1.
\]
Here, the positive number $c(A,\bfP)$ depends at most on $A$ and $\bfP$. Notice here 
that, whilst the set $\calA$ has elements of size roughly $X$, the elements of $\calB$ have 
size roughly a power of $\log X$. This reduction in size is crucial to our condensation 
argument.\par

By iterating these two steps sufficiently many times, much as was done in \S3 in the simpler 
linear case in the proofs of Theorems \ref{theorem3.1} and \ref{theorem3.2}, we 
ultimately obtain a set $\calC$ algebraically Freiman $\bfP$-isomorphic to $\calA$, and 
satisfying the property that
\[
\text{Env}(\calC)\ll_{A,\bfP}1\quad \text{and}\quad d(\calB)\ll_{A,\bfP}1.
\]
All that remains is to take care in controlling the behaviour in these results of the implicit 
constants depending on $A$ and $\bfP$.\par

We describe the details of the argument just outlined in the next section. For now, it suffices 
to say that in the setting countenanced in the above discussion, we are able to show that, 
subject to the validity of the Generalised Riemann Hypothesis for all Dedekind zeta 
functions, there is a set of algebraic numbers $\calC$ algebraic Freiman $\bfP$-isomorphic 
to $\calA$ with
\[
d(\calC)\le \exp_2(\kap_1A)\quad \text{and}\quad \text{Env}(\calC)\le \exp_4(\kap_2A),
\]
where $\kap_1$ and $\kap_2$ are positive numbers depending at most on $\bfP$.

\section{Condensations for non-linear systems of equations, II}
Let us now put the plan of the previous section into action. It is worth stressing that our 
bounds will be extraordinarily weak. In consequence, it makes sense to avoid stress on 
detailed bounds, but instead opt for estimates somewhat weaker than might be achieved 
with greater attention to detail, but ones nonetheless simple to state in suitable notation.

\par Let $s$ and $r$ be natural numbers with $s\ge 2$, and for $1\le i\le r$ consider fixed 
polynomials $P_i=P_i(\bfx)\in \dbZ[x_1,\ldots ,x_s]$ of respective degrees $t_i$. We write 
$\| P_i\|_1$ for the sum of the absolute values of the coefficients of $P_i$, and we suppose 
that $t_i\le t$ and $\|P_i\|_1\le k$ for $1\le i\le r$. Then, in the sense of Grosu 
\cite{Gro2014}, the polynomial system $\bfP$ is $(k,t)$-bounded. Note that, in view of our 
work in \S3, there is no loss of generality in supposing that $t\ge 2$. Our initial discussion 
will be focused on establishing the iterative step described in the previous section. Suppose 
then that $\calA$ is a set of algebraic numbers with $\text{card}(\calA)=A\ge 2$. Our 
discussion will be simplified by introducing the function
\begin{equation}\label{5.1}
\nu(n)=(2t)^{(2t)^{2^n}}.
\end{equation}
To avoid any potential ambiguity, we note that this is a $3$-fold iterated exponential function 
of $n$. Equipped with this notation, it is convenient to suppose that
\begin{equation}\label{5.2}
d(\calA)\le (2t)^{2^A}\quad \text{and}\quad \text{Env}(\calA)=X.
\end{equation}

In most familiar applications of algebraic number theory, analytic number theorists are used 
to working with a fixed number field $K$ wherein the degree and discriminant are 
well-controlled. Unfortunately for us, we require discussions of field extensions of $\dbQ$ 
with enormous degree and discriminant, and so we are forced to pay attention to details 
that in normal circumstances would not delay our argument. Our initial focus lies on the 
non-zero algebraic number
\begin{equation}\label{5.3}
\Ups(\calA)=\Biggl( \prod_{\substack{a_1,a_2\in \calA \\ a_1\ne a_2}}(a_1-a_2)\Biggr) 
\Biggl( \prod_{i=1}^r\prod_{\substack{\bfa\in \calA^s\\ P_i(\bfa)\ne 0}}P_i(\bfa)\Biggr) .
\end{equation}
We seek a rational prime number $\pi$ with properties associated to $\Ups(\calA)$ outlined 
in the discussion of the previous section. In preparation for our application of the 
Chebotarev density theorem, we discuss the Galois closure $K^c$ of $K=\dbQ(\calA)$, and 
some of its properties.

\begin{lemma}\label{lemma5.1}
One has $[K^c:\dbQ]\le \nu(A+1)$.
\end{lemma}

\begin{proof} By the primitive element theorem, there is some algebraic number 
$\tet \in K$ for which $K=\dbQ(\tet)$. The minimal polynomial $m_\tet$ of $\tet$ over 
$\dbZ$ has degree $[K:\dbQ]=d(\calA)\le (2t)^{2^A}$. Thus, the splitting field of $m_\tet$, 
which contains $K^c$, has degree at most $d(\calA)!$. We therefore conclude that
\[
[K^c:\dbQ]\le d(\calA)^{d(\calA)}\le (2t)^B,
\]
where
\[
B=2^A(2t)^{2^A}\le (2t)^{A+2^A}\le (2t)^{2^{A+1}}.
\]
Thus, on recalling the notation \eqref{5.1}, we find that $[K^c:\dbQ]\le \nu(A+1)$, and the 
proof of the lemma is complete.
\end{proof}

\begin{lemma}\label{lemma5.2}
One has $\text{\rm Disc}(K^c:\dbQ)\le (2tX)^{\nu (A+4)}$.
\end{lemma}

\begin{proof} We begin by considering a typical element $a\in \calA$. Suppose that 
$[\dbQ(a):\dbQ]=d$, whence the minimal polynomial $m_a$ of $a$ over $\dbZ$ has degree 
$d$. We note for future reference that $d\le [K:\dbQ]=d(\calA)$. Let $S(a)$ denote the 
splitting field for $m_a$ over $\dbQ$, so that $S(a)=\dbQ(\bet_1,\ldots ,\bet_d)$ for some 
distinct algebraic numbers $\bet_1,\ldots ,\bet_d$. We put $S_0(a)=\dbQ$, and when 
$1\le j\le d$ we define
\[
S_j(a)=\dbQ(\bet_1,\ldots ,\bet_j).
\]
For each index $j$ with $1\le j\le d$, we have
\begin{equation}\label{5.4}
\text{Disc}(\dbQ(\bet_j):\dbQ)\le \text{Disc}(m_{\bet_j})=\text{Disc}(m_a).
\end{equation}
Recall the upper bounds \eqref{5.2}. Then, by considering the resultant of $m_a$ and 
$m'_a$ in terms of the determinant of the associated Sylvester matrix, noting that the 
coefficients of $m_a$ are bounded in absolute value by $\text{Env}(\calA)$, we see that
\begin{align}
\text{Disc}(m_a)&\le \left( \text{deg}(m_a)\right)!
\left( \text{Env}(\calA)\right)^{2\text{deg}(m_a)}\notag \\
&\le d(\calA)^{d(\calA)}X^{2d(\calA)}\notag \\
&\le \nu(A+1)X^{2(2t)^{2^A}}.\label{5.5}
\end{align}

\par Now observe that, as a consequence of a simple bound of T\^oyoma \cite{Toy1955}, 
whenever $E:\dbQ$ and $F:\dbQ$ are two field extensions, then
\begin{equation}\label{5.6}
\text{Disc}(EF:\dbQ)\le \left( \text{Disc}(E:\dbQ)\,\text{Disc}(F:\dbQ)\right)^{[EF:\dbQ]}.
\end{equation}
Thus, for $1\le j<d$, it follows from \eqref{5.4} that
\begin{align}
\text{Disc}(S_{j+1}(a):\dbQ)&\le \left( \text{Disc}(S_j(a):\dbQ)\,\text{Disc}
(\dbQ(\bet_{j+1}):\dbQ)\right)^{[K^c:\dbQ]}\notag \\
&\le \left( \text{Disc}(S_j(a):\dbQ)\,\text{Disc}(m_a)\right)^{[K^c:\dbQ]}.\label{5.7}
\end{align}
Since it also follows from \eqref{5.4} that $\text{Disc}(S_1(a):\dbQ)\le \text{Disc}(m_a)$, 
we may apply the relation \eqref{5.7} inductively to derive the relation
\begin{equation}\label{5.8}
\text{Disc}(S_d(a):\dbQ)\le \left( \text{Disc}(m_a)\right)^{d[K^c:\dbQ]^{d-1}}.
\end{equation}
We therefore deduce from \eqref{5.5} and Lemma \ref{lemma5.1} that
\[
\text{Disc}(S(a):\dbQ)\le \left( \nu(A+1)X^{2(2t)^{2^A}}\right)^{\nu(A+1)^{d(\calA)}}.
\]
A modicum of computation confirms that
\[
\nu(A+1)^{d(\calA)}\le (2t)^{(2t)^{2^{A+1}}(2t)^{2^A}}\le \nu(A+2),
\]
whilst
\[
\nu(A+1)^{\nu(A+2)}\le (2t)^{\nu(A+3)}\quad \text{and}\quad 2(2t)^{2^A}\nu(A+2)\le 
\nu(A+3).
\]
Consequently, we arrive at the simplified upper bound
\begin{equation}\label{5.9}
\text{Disc}(S(a):\dbQ)\le (2tX)^{\nu(A+3)}.
\end{equation}

\par At this point, we have bounded the discriminant associated to only one element of 
$\calA$. The Galois closure $K^c$ of $\dbQ(\calA)$, however, is the compositum of all the 
splitting fields $S(a)$ for $a\in \calA$. We therefore apply the relation \eqref{5.6} as in the 
deduction of \eqref{5.8} to establish the bound
\[
\text{Disc}(K^c:\dbQ)\le 
\left( \max_{a\in \calA}\text{Disc}(S(a):\dbQ)\right)^{A[K^c:\dbQ]^{A-1}}.
\]
By Lemma \ref{lemma5.1}, one has
\[
A[K^c:\dbQ]^{A-1}\le \nu(A+1)^A,
\]
so the upper bound \eqref{5.9} yields the estimate
\[
\text{Disc}(K^c:\dbQ)\le (2tX)^B,
\]
where
\[
B=\nu(A+3)\nu(A+1)^A=(2t)^{A(2t)^{2^{A+1}}+(2t)^{2^{A+3}}}
\le (2t)^{(2t)^{2^{A+4}}}=\nu(A+4).
\]
Thus we conclude that $\text{Disc}(K^c:\dbQ)\le (2tX)^{\nu(A+4)}$, completing the proof.
\end{proof}

We shall need to identify a rational prime number $\pi$ having the property that the 
algebraic number $\Ups(\calA)$ is a unit modulo $\pi$. Let $\Ups_0(\calA)$ denote the least 
positive (rational) integer having the property that the algebraic number
\[
\Ups_1(\calA)=\Ups_0(\calA)\Ups(\calA)
\]
is an algebraic integer. Then, by taking norms, it is evident that it suffices to arrange that 
$\pi$ does not divide the rational integer
\[
\Ups^*(\calA)=|N_{K^c:\dbQ}(\Ups_0(\calA))N_{K^c:\dbQ}(\Ups_1(\calA))|.
\]

\begin{lemma}\label{lemma5.3} One has 
$1\le \Ups^*(\calA)\le (2kX)^{r(2A)^s\nu(A+2)}$.
\end{lemma}

\begin{proof} We begin by taking a crude approach to bounding 
$|N_{K^c:\dbQ}(\Ups(\calA))|$, applying bounds for the complex absolute values of the 
conjugates of each element of $\calA$. Let $a$ be a typical element of $\calA$. Since 
$\text{Env}(\calA)=X$, the minimal polynomial $m_a$ of $a$ over $\dbZ$ satisfies the 
relation $\| m_a\|_1\le X$. Also, since $\text{deg}(m_a)\le d(\calA)$, we find that $m_a$ 
takes the form
\begin{equation}\label{5.10}
b_0(a)x^d+\ldots +b_{d-1}(a)x+b_d(a),
\end{equation}
in which $d\le d(\calA)$ and $|b_0(a)|\ge 1$. The (complex) absolute value of $a$ therefore 
satisfies either the upper bound $|a|\le 1$, or else is constrained by the inequality
\[
|b_0(a)a^d|\le |a|^{d-1}\|m_a\|_1\le |a|^{d-1}\text{Env}(\calA)=|a|^{d-1}X,
\]
whence $|a|\le X$. Thus, in any case, one has $|a|\le X$. Since the conjugates of $a$ are 
also roots of the polynomial $m_a$, one sees in this way that every conjugate of $a$ in 
$K^c$ has absolute value bounded above by $X$.\par

Recall the formula \eqref{5.3}. It follows from our discussion thus far that the element 
$\Ups(\calA)$ of $K^c$ satisfies the bound
\[
|\Ups(\calA)|\le (2X)^{A^2}(kX^t)^{rA^s}.
\]
Here, we have made use of the observation that, since each polynomial $P_i(\bfx)$ is 
$(k,t)$-bounded for $1\le i\le r$, then for $\bfa\in \calA^s$ one has
\[
|P_i(\bfa)|\le \|P_i\|_1\Bigl( \max_{1\le i\le s}|a_i|\Bigr)^{t_i}\le kX^t.
\]
In order to bound the norm of $\Ups(\calA)$, we must multiply all of the conjugates of 
$\Ups(\calA)$ together. However, since we assume that $s\ge 2$, the concluding remark of 
the preceding paragraph shows that each of these conjugates is bounded above by
\[
(2X)^{A^2}(kX^t)^{rA^s}\le (2kX)^{2rtA^s}.
\]
Thus, multiplying all of these conjugates together, we find that
\begin{equation}\label{5.11}
|N_{K^c:\dbQ}(\Ups(\calA))|\le \left( (2kX)^{2rtA^s}\right)^{[K^c:\dbQ]}.
\end{equation}

\par Next, we investigate the denominator $\Ups_0(\calA)$. Referring to the minimal 
polynomial \eqref{5.10} of $a$ over $\dbZ$, we see that $b_0(a)a$ is an algebraic integer 
and $|b_0(a)|\le X$. An inspection of \eqref{5.3} therefore shows that $\Ups_0(\calA)$ 
is a positive rational integer dividing
\[
\Biggl( \prod_{a_1,a_2\in \calA}b_0(a_1)b_0(a_2)\Biggr) 
\Biggl( \prod_{i=1}^r\prod_{\bfa\in \calA^s}b_0(a_1)^t\cdots b_0(a_s)^t\Biggr) .
\]
Thus, we have
\[
\Ups_0(\calA)\le X^{2A^2+rstA^s},
\]
whence
\[
|N_{K^c:\dbQ}(\Ups_0(\calA))|\le \left( X^{(s+2)rtA^s}\right)^{[K^c:\dbQ]}.
\]
By combining this estimate together with \eqref{5.11}, we therefore discern that
\begin{align}
\Ups^*(\calA)&=|N_{K^c:\dbQ}(\Ups_0(\calA))^2N_{K^c:\dbQ}(\Ups(\calA))| \notag \\
&\le \left( (2kX)^{(2s+6)rtA^s}\right)^{[K^c:\dbQ]}.\label{5.12}
\end{align}

\par Finally, we simplify the estimate supplied by \eqref{5.12}. Observe first that, since 
$s\ge 2$, we have $2s+6\le 2^{s+2}$. Thus, on applying Lemma \ref{lemma5.1}, we infer 
that
\[
\Ups^*(\calA)\le \left( 2kX\right)^{4rt(2A)^s\nu(A+1)}. 
\]
However, one has $4t\nu(A+1)\le \nu(A+2)$. We consequently conclude that
\[
\Ups^*(\calA)\le (2kX)^{r(2A)^s\nu(A+2)}.
\]
On noting that $\Ups^*(\calA)$ is non-zero, by construction, the proof of the lemma is 
complete.
\end{proof}

It is now time to select the rational prime number $\pi$ by applying the Chebotarev density 
theorem. By the primitive element theorem, there exists an element $\Tet\in K^c$ having 
the property that $K^c=\dbQ(\Tet)$. It is apparent, moreover, that by making an 
appropriate choice for $\Tet$, we may assume not only that $\calA\subset \dbZ[\Tet]$, but 
also that all of the conjugates of the elements of $\calA$ within $K^c$ lie in $\dbZ[\Tet]$. 
With this choice for $\Tet$ now fixed in such a manner, we seek a rational prime number 
$\pi$ with $\pi\nmid \Ups^*(\calA)$ having the property that the minimal polynomial 
$m_\Tet$ of $\Tet$ over $\dbZ$ splits into linear factors modulo $\pi$. This allows us to 
bijectively map the set $\calA$ into a set of residues modulo $\pi$, while preserving the 
salient features of the solution set associated with the system of polynomial equations 
$\bfP={\mathbf 0}$. Throughout, we abbreviate Generalised Riemann Hypothesis to GRH.

\begin{lemma}\label{lemma5.4}
There exists an effectively computable positive absolute constant $M_1$ with the following 
property. Suppose that GRH holds for the Dedekind zeta function associated with the field 
extension $K^c:\dbQ$. In addition, assume that
\begin{equation}\label{5.13}
Y\ge r(2A)^s(2t)^{\nu(A+3)}\log (2kX)
\end{equation}
and
\begin{equation}\label{5.14}
Y\ge M_1(2t)^{2\nu(A+4)}(\log (2tX))^4.
\end{equation}
Then there exists a rational prime number $\pi$, with $Y<\pi\le 16Y$ and 
$\pi \nmid \Ups^*(\calA)$, having the property that the minimal polynomial $m_\Tet$ of 
$\Tet$ over $\dbZ$ splits into linear factors over $\dbF_\pi[t]$.
\end{lemma}

\begin{proof} We work under the hypotheses \eqref{5.13} and \eqref{5.14} throughout. An 
effective version of the Chebotarev density theorem is provided by Lagarias and Odlyzko 
under the assumption of GRH for the Dedekind zeta function associated with the field 
extension $K^c:\dbQ$. Denote by $\pi_\Tet (x)$ the number of rational prime numbers $p$ 
with $p\le x$ having the property that $m_\Tet$ splits into linear factors over $\dbF_p$. Put
\[
G=\text{Gal}(K^c:\dbQ),\quad n=[K^c:\dbQ]\quad \text{and}\quad 
D=\text{Disc}(K^c:\dbQ).
\]
Then it follows from \cite[Theorem 1.1]{LO1977} that there exists a positive absolute 
constant $M_0$ such that
\begin{equation}\label{5.15}
\biggl| \pi_\Tet (x)-\frac{\text{Li}(x)}{|G|}\biggr| \le M_0\biggl( 
\frac{x^{1/2}\log (Dx^n)}{|G|}+\log D\biggr) .
\end{equation}
Here, we have written $\text{Li}(x)$ for the usual logarithmic integral function.\par

On recalling Lemmata \ref{lemma5.1} and \ref{lemma5.2}, we find that when $x\le X$, we 
have
\begin{align*}
\log (Dx^n)&\le \nu(A+4)\log (2tX)+\nu(A+1)\log x\\
&\le 2\nu(A+4)\log (2tX).
\end{align*}
Moreover, it follows from a trivial upper bound for $|G|$ together with Lemma 
\ref{lemma5.1} that
\begin{equation}\label{5.16}
|G|\le [K^c:\dbQ]!\le [K^c:\dbQ]^{[K^c:\dbQ]}\le \nu(A+1)^{\nu(A+1)}\le 
(2t)^{\nu (A+2)},
\end{equation}
whence
\begin{align*}
|G|\log D&\le (2t)^{\nu (A+2)}\nu(A+4)\log (2tX)\\
&\le (2t)^{\nu(A+3)}\log (2tX).
\end{align*}
Thus, we deduce from \eqref{5.15} that
\[
\biggl| \pi_\Tet (x)-\frac{\text{Li}(x)}{|G|}\biggr| \le \frac{M_0}{|G|}\left( 
2x^{1/2}\nu(A+4)\log (2tX)+(2t)^{\nu(A+3)}\log (2tX)\right) .
\]

\par Suppose that $M_1$ is sufficiently large in terms of $M_0$. Then, under the hypothesis 
\eqref{5.14}, we have
\begin{align*}
\pi_\Tet (16Y)&\ge \frac{1}{|G|}\Bigl( \frac{16Y}{\log (16Y)}-8M_0Y^{1/2}\nu(A+4)
\log (2tX)-M_0(2t)^{\nu(A+3)}\log (2tX)\Bigr) \\
&\ge \frac{1}{|G|}\Bigl( \frac{16Y}{\log (16Y)}-\frac{4Y}{\log Y}\Bigr) \ge 
\frac{8Y}{|G|\log Y}.
\end{align*}
Meanwhile, in a similar manner one finds that
\begin{align*}
\pi_\Tet (Y)&\le \frac{1}{|G|}\Bigl( \frac{2Y}{\log Y}+2M_0Y^{1/2}\nu(A+4)\log (2tX)+
M_0(2t)^{\nu(A+3)}\log (2tX)\Bigr) \\
&\le \frac{1}{|G|}\Bigl( \frac{2Y}{\log Y}+\frac{2Y}{\log Y}\Bigr) =\frac{4Y}{|G|\log Y}.
\end{align*}
Thus, we discern that
\[
\pi_\Tet (16Y)-\pi_\Tet (Y)\ge \frac{4Y}{|G|\log Y}.
\]

\par Let $\Pi$ denote the set of rational prime numbers $p$ with $Y<p\le 16Y$ for which 
$m_\Tet$ splits into linear factors over $\dbF_p$. Then
\[
\prod_{p\in \Pi}p\ge Y^{4Y/(|G|\log Y)}=\exp \left( 4Y/|G|\right) .
\]
Meanwhile, from Lemma 5.3 we find that
\[
\log \Ups^*(\calA)\le r(2A)^s\nu(A+2)\log (2kX).
\]
Thus, recalling the upper bound \eqref{5.16} for $|G|$, we have
\[
\prod_{p\in \Pi}p>\Ups^*(\calA)
\]
provided only that
\begin{equation}\label{5.17}
\frac{4Y}{(2t)^{\nu(A+2)}}>r(2A)^s\nu(A+2)\log (2kX).
\end{equation}
But $\nu(A+2)(2t)^{\nu(A+2)}\le (2t)^{\nu(A+3)}$, and so the hypothesis \eqref{5.13} is 
sufficient to ensure the validity of \eqref{5.17}. With this condition now satisfied, we 
conclude that there exists a rational prime number $\pi$ with $Y<\pi\le 16Y$ satisfying 
$\pi\nmid \Ups^*(\calA)$, and such that $m_\Tet$ splits into linear factors over $\dbF_\pi$. 
The conclusion of the lemma follows.
\end{proof}

We are now in a position to move on to the second step in the inductive phase of the 
argument, applying the method of Grosu \cite{Gro2014}. The conclusion of Lemma 
\ref{lemma5.4} shows that there is a rational prime number $\pi$ with
\[
\pi\le 16M_1r(2A)^s(2t)^{2\nu(A+4)}(\log (2tkX))^4
\]
having the property that $\pi \nmid \Ups^*(\calA)$, and such that $m_\Tet$ splits into linear 
factors over $\dbF_\pi$. Let $a_0$ be any zero of the polynomial $m_\Tet$ in $\dbF_\pi$. 
Since $\calA\subset \dbZ[\Tet]$, the ring homomorphism 
$\Phi :\dbZ[\Tet]\rightarrow \dbF_\pi$ defined by putting $\Phi(\Tet)=a_0$ restricts to a 
well-defined map $\varphi: \dbZ[\calA]\rightarrow \dbF_\pi$.\par

We claim that the set $\calA$ has image $\calB=\varphi(\calA)$ in which, for pairs of 
elements $a_1,a_2\in \calA$, one has $\varphi(a_1)=\varphi(a_2)$ if and only if 
$a_1=a_2$. This claim will be confirmed by verifying that whenever $a_1\ne a_2$, then 
$\varphi(a_1)\ne \varphi(a_2)$. By way of seeking a contradiction, suppose that $a_1\ne 
a_2$ and $\varphi(a_1)=\varphi(a_2)$. Then we have $\Phi(a_1)=\Phi(a_2)$, and the 
homomorphism property of $\Phi$ implies that $\Phi(a_1-a_2)=0$. But $\Ups^*$ is a 
multiple of $a_1-a_2$, say $\Ups^*=\mu(a_1-a_2)$ for a suitable element $\mu$ of 
$\dbZ[\Tet]$. The homomorphism property of $\Phi$ thus ensures that 
$\Phi(\Ups^*)=\Phi(\mu)\Phi(a_1-a_2)=0$. However, we have $\Ups^*\in \dbZ$, and since 
$\pi\nmid \Ups^*$ we find that $0=\Phi(\Ups^*)=\Ups^*\Phi(1)$ and hence $\Phi(1)=0$. 
This contradicts the homomorphism property of $\Phi$, confirming our earlier claim.\par

We also claim that, for $1\le i\le r$ and $\bfa\in \calA^s$, one has $P_i(\bfa)=0$ if and only 
if $P_i(\varphi(\bfa))=0$. In this instance it suffices to show that when $P_i(\bfa)\ne 0$, 
then $P_i(\varphi(\bfa))\ne 0$. We again proceed by seeking a contradiction, assuming that 
$P_i(\bfa)\ne 0$ and $P_i(\varphi(\bfa))=0$. Then the homomorphism property of $\Phi$ 
ensures that $\Phi(P_i(\bfa))=0$. But $\Ups^*$ is a multiple of $P_i(\bfa)$, say 
$\Ups^*=\mu'P_i(\bfa)$ for a suitable element $\mu'$ of $\dbZ[\Tet]$. Thus, in a similar 
manner to that described in the previous paragraph, we find that 
$\Phi(\Ups^*)=\Phi(\mu')\Phi(P_i(\bfa))=0$, contradicting the fact that $\pi\nmid \Ups^*$. 
This contradiction again confirms our claim.

\begin{lemma}\label{lemma5.5}
Suppose that $\pi$ is a prime number having the property that $\pi\nmid \Ups^*(\calA)$, 
and such that $m_\Tet$ splits into linear factors over $\dbF_\pi$. Suppose also that
\begin{equation}\label{5.18}
\pi>(2kt)^{(2t)^{2^{A+1}}}.
\end{equation}
Then there exists an algebraic extension $L$ of $\dbQ$ of degree at most $(2t)^{2^A}$, 
and a subset $\calC\subset L$ with $\text{\rm card}(\calC)=A$, having the following 
properties:
\begin{enumerate}
\item[(a)] there is an injective map $\ome:\calC\rightarrow \dbF_\pi$, with 
$\ome(\calC)=\varphi(\calA)$, having the property that the canonical induced map 
${\widetilde \ome}:\dbZ[\calC]\rightarrow \dbF_\pi$ is a ring homomorphism;
\item[(b)] given $\bfa\in \calA^s$, define $c_i=\ome^{-1}(\varphi(a_i))$ $(1\le i\le s)$. 
Then one has $P_i(\bfc)=0$ $(1\le i\le r)$ for $\bfc\in \calC^s$ if and only if 
$P_i(\varphi(\bfa))=0$ for $\bfa\in \calA^s$;
\item[(c)] one has $\text{\rm Env}(\calC)\le \nu(A+1)\pi$.
\end{enumerate} 
\end{lemma}

In order to avoid ambiguity, we stress that the map ${\widetilde \ome}$ is defined for 
$f(\bfx)\in \dbZ[x_1,\ldots ,x_m]$ by taking 
\[
{\widetilde \ome}\left(f(c_1,\ldots ,c_m)\right)=f(\ome(c_1),\ldots ,\ome(c_m)).
\]

\begin{proof}[The proof of Lemma \ref{lemma5.5}]
The desired conclusion is a consequence of the argument of Grosu 
\cite[Lemma 8.1]{Gro2014}, though care is required in interpreting the argument underlying 
the latter proof so as to obtain the desired outcome. Following the general strategy of 
Grosu, we assign distinct indeterminates $x_a$ to each element $\varphi(a)$ of 
$\calB=\varphi(\calA)$. Certain equations are then known to have solutions over $\dbF_\pi$, 
specifically
\begin{equation}\label{5.19}
P_i(x_{a_1},\ldots ,x_{a_s})=0\quad (1\le i\le r)
\end{equation}
has a solution $(x_{a_1},\ldots ,x_{a_s})=(\varphi(a_1),\ldots ,\varphi(a_s))$ whenever 
\[
P_i(a_1,\ldots ,a_s)=0\quad (1\le i\le r)
\]
for $\bfa\in \calA^s$. In addition, one has certain non-equations. First, of course, one has
\[
x_{a_1}-x_{a_2}\ne 0
\]
whenever $x_{a_1}=\varphi(a_1)$ and $x_{a_2}=\varphi(a_2)$ for $a_1\ne a_2$ with 
$a_1,a_2\in \calA$. Moreover, we have
\[
P_i(x_{a_1},\ldots ,x_{a_s})\ne 0
\]
whenever $x_{a_j}=\varphi(a_j)$ $(1\le j\le s)$ for $(a_1,\ldots ,a_s)\in \calA^s$ with 
$P_i(a_1,\ldots ,a_s)\ne 0$. Taken together, we now have a list of equations and 
non-equations in the variables $x_a$ $(a\in \calA)$, all defined over $\dbF_\pi$, and with 
the defining equations all $(k,t)$-bounded. It is worth emphasising, for the uninitiated, that 
the number of equations here may be very large. When $r=1$, for example, the number of 
equations may be as large, roughly, as $A^{s-1}$.\par

The argument of the proof of Grosu \cite[Lemma 8.1]{Gro2014} now shows via an 
elimination procedure using resultants that over the algebraic closure ${\overline \dbQ}$ of 
$\dbQ$, the equations \eqref{5.19} possess a solution if and only if certain eliminant 
polynomials are constant and equal to $0$. By applying the same elimination procedure over 
$\dbF_\pi$, however, one sees that these constant polynomials must be $0$ over 
$\dbF_\pi$, since the equations \eqref{5.19} possess the solution $x_a=\varphi(a)$ 
$(a\in \calA)$ in that setting. Provided that these eliminant polynomials have small enough 
coefficients in terms of $\pi$, therefore, one finds that in the setting of ${\overline \dbQ}$, 
these eliminant polynomials are indeed $0$, and hence the system \eqref{5.19} possesses 
a solution, say $x_a=\psi(a)$ $(a\in \calA)$, lying in ${\overline \dbQ}$.\par

This is not the end of the story. It is shown first by Grosu \cite[Lemma 8.1]{Gro2014} that 
the field $L=\dbQ(\psi(\calA))$ has degree at most ${\tilde t}=(2t)^{2^A}$ over $\dbQ$. 
Second, all of the eliminant polynomials to which we alluded above have coefficients 
bounded by
\[
{\tilde k}=(2kt)^{(2t)^{2^{A+1}}},
\]
and thus the condition \eqref{5.18} suffices for the desired conclusion. Indeed, Grosu 
shows that the eliminant polynomials are all $({\tilde k},{\tilde t})$-bounded. Third, the 
elements $\psi(a)$ $(a\in \calA)$ may be chosen in such a manner that there is a ring 
homomorphism $\gam: \dbZ[\calC]\rightarrow \dbF_\pi$ which sends $\psi(a)$ to 
$\varphi(a)$ for each $a\in \calA$. A subtle detail of this last conclusion is that it may be 
necessary to take $\psi(a)$ to be a rational integer $b$ lying in the set 
$\{0,1,\ldots ,\pi-1\}$ with $b\equiv \varphi(a)\mmod{\pi}$ (see the fifth and sixth 
paragraphs of Step 2 of the proof of \cite[Lemma 8.1]{Gro2014}).\par

We are now in possession of sufficient detail to complete the proof of the lemma. We have 
already confirmed the claims made in the statement of the lemma concerning the existence 
of $L$, the subset $\calC$, the degree of $L$ over $\dbQ$, and we have also explained the 
hypothesis \eqref{5.18}. We define $\ome:\calC\rightarrow \dbF_\pi$ by taking 
$\ome(c)=\gam(c)$ for $c\in \calC$, and then ${\widetilde \ome}$ coincides with $\gam$ 
by virtue of the ring homomorphism property of $\gam$. Moreover, if 
$\ome(c_1)=\ome(c_2)$ for some elements $c_1$ and $c_2$ of $\calC$, then 
$\gam(c_1)=\gam(c_2)$. When $c_i=\psi(a_i)$ $(i=1,2)$ with $a_1,a_2\in \calA$, then we 
have
\[
\varphi(a_1)=\gam(\psi(a_1))=\gam(c_1)=\gam(c_2)=\gam(\psi(a_2))=\varphi(a_2).
\]
Thus, from the properties of the mapping $\varphi$, we have $a_1=a_2$. Hence
\[
c_1=\psi(a_1)=\psi(a_2)=c_2.
\]
The mapping $\ome$ is consequently injective. This confirms the claim (a).\par

By construction, if $c_i=\psi(a_i)$ for $1\le i\le s$ and $\bfa\in \calA^s$, then
\[
P_i(\bfc)=0\quad (1\le i\le r)
\]
if and only if
\[
{\widetilde \ome}(P_i(\bfc))=0\quad (1\le i\le r).
\]
But 
\[
{\widetilde \ome}(P_i(\bfc))=P_i(\ome(c_1),\ldots ,\ome(c_s))=
P_i(\varphi(a_1),\ldots ,\varphi(a_s))=0\quad (1\le i\le r).
\]
Thus $P_i(\bfc)=0$ $(1\le i\le r)$ for $\bfc\in \calC^s$ if and only if $P_i(\varphi(\bfa))=0$ 
$(1\le i\le r)$. This confirms the claim (b).\par

Finally, each element $c\in \calC$ is either an integer from the set $\{0,1,\ldots ,\pi-1\}$, or 
else satisfies a $({\tilde k},{\tilde t})$-bounded polynomial of degree at most $(2t)^{2^A}$ 
having integral coefficients of absolute value at most
\[
(2kt)^{(2t)^{2^{A+1}}}\le \pi .
\]
In the latter case, the minimal polynomial $m_c$ of $c$ over $\dbZ$ is a divisor of a 
polynomial $q\in \dbZ[t]$, with $\|q\|_1\le {\tilde k}\le \pi$ and 
$\text{deg}(q)\le (2t)^{2^A}$. If $\text{deg}(q)=d$ and we write 
$q(x)=q_dx^d+\ldots +q_1x+q_0$ with $q_i\in \dbZ$ $(0\le i\le d)$, then we have
\[
\sum_{l=0}^d|q_l|^2\le \biggl( \sum_{l=0}^d|q_l|\biggr)^2=\|q\|_1^2.
\]
It is then a consequence of the corollary to the main theorem of Granville \cite{Gra1990} 
that if $r\in \dbZ[x]$ is any polynomial divisor of $q$, then
\[
\|r\|_2\le \Bigl( \frac{\sqrt{5}+1}{2}\Bigr)^d\|q\|_2\le \Bigl( \frac{\sqrt{5}+1}{2}\Bigr)^d
\|q\|_1\le \Bigl( \frac{\sqrt{5}+1}{2}\Bigr)^d \pi.
\]
By Cauchy's inequality, therefore, we have
\[
\|r\|_1\le (\text{deg}(r)+1)^{1/2}\|r\|_2\le 2d\Bigl( \frac{\sqrt{5}+1}{2}\Bigr)^d\pi.
\]
Hence, every element $c\in \calC$ has minimal polynomial $m_c$ over $\dbZ$ having 
degree at most $(2t)^{2^A}$, with
\[
\|m_c\|_1\le 2(2t)^{2^A}2^{(2t)^{2^A}}\pi\le \nu(A+1)\pi .
\]
Thus $\text{Env}(\calC)\le \nu(A+1)\pi$, completing the proof of claim (c).
\end{proof}

The plan outlined in the previous section may now be applied to good effect. Suppose that 
$\calA$ is a set of algebraic numbers with
\[
d(\calA)\le (2t)^{2^A}\quad \text{and}\quad \text{Env}(\calA)=X.
\]
Then, assuming GRH for all Dedekind zeta functions, it follows from Lemma \ref{lemma5.4} 
that we can find a rational prime number 
$\pi$ with
\[
\pi>(2kt)^{(2t)^{2^{A+1}}}
\]
and
\begin{equation}\label{5.20}
\pi\le \max\{ 16(2kt)^{(2t)^{2^{A+1}}},16M_1rA^{2s}(2t)^{2\nu(A+4)}(\log (2tkX))^4\}
\end{equation}
such that $\pi\nmid \Ups^*(\calA)$, and having the property that $m_\Tet$ splits into linear 
factors over $\dbF_\pi$. As a consequence of Lemma \ref{lemma5.5}, we deduce that 
there is a set $\calC$ of algebraic numbers having the property that
\[
d(\calC)\le (2t)^{2^A}\quad \text{and}\quad \text{Env}(\calC)\le \nu(A+1)\pi,
\]
and having the property, moreover, that there is a bijection $\psi:\calA\rightarrow \calC$ 
which is an algebraic Freiman $\bfP$-isomorphism.\par

We may now iterate this step, starting with the set of algebraic numbers $\calC$, and 
deriving a new set $\calC'$ algebraic Freiman $\bfP$-isomorphic to $\calC$, and with
\[
d(\calC')\le (2t)^{2^A}\quad \text{and}\quad \text{Env}(\calC')\le \nu(A+1)\pi',
\]
where
\[
\pi'\le \max\{ 16(2kt)^{(2t)^{2^{A+1}}},16M_1rA^{2s}(2t)^{2\nu(A+4)}
(\log (2tk\nu(A+1)\pi ))^4\}.
\]
The composition of the two algebraic Freiman $\bfP$-isomorphisms that we have 
encountered here provides an algebraic Freiman $\bfP$-isomorphism from $\calA$ to 
$\calC'$. Plainly, it makes sense to iterate this process repeatedly so long as the associated 
algebraic enveloping radius is decreasing.\par

In order to assess the strength of the ensuing bounds, it makes sense to simplify our 
estimates so as to make iteration tractable. Observe first that the bound \eqref{5.20} may 
be simplified by noting that, when it is satisfied, one has
\[
\pi\le 16M_1rA^{2s}(2kt)^{2\nu(A+4)}(\log (2tkX))^4,
\]
whence
\begin{align*}
\nu(A+1)\pi&\le 16M_1rA^{2s}(2kt)^{2\nu(A+4)}\nu(A+1)(\log (2tkX))^4\\
&\le M_1rA^{2s}(2kt)^{\nu(A+5)}(\log (2tkX))^4.
\end{align*}
We therefore have
\[
\text{Env}(\calC)\le \nu(A+1)\pi \le \tfrac{1}{2}X=\tfrac{1}{2}\text{Env}(\calA)
\]
provided that $X$ is large and
\[
X\ge M_1^2r^2A^{4s}(2kt)^{2\nu(A+5)+1}.
\]
Indeed, provided that $X$ is large enough, one has
\[
(\log(2tkX))^4\le \tfrac{1}{2}\sqrt{tkX},
\]
and hence
\[
M_1rA^{2s}(2kt)^{\nu(A+5)}(\log (2tkX))^4\le \tfrac{1}{2}M_1rA^{2s}
(2kt)^{\nu(A+5)}(tk)^{1/2}X^{1/2}\le \tfrac{1}{2}X.
\]
Thus, by iterating this condensation process, we may ensure that $\calA$ is algebraic 
Freiman $\bfP$-isomorphic to a set $\calB$ of algebraic numbers with
\[
d(\calB)\le (2t)^{2^A}\quad \text{and}\quad \text{Env}(\calB)\le 
M_1^2r^2A^{4s}(2kt)^{\nu(A+6)}.
\]
We summarise these deliberations in the form of a theorem.

\begin{theorem}\label{theorem5.6} Assume GRH for all Dedekind zeta functions. Let 
\[
P_i(\bfx)\in \dbZ[x_1,\ldots ,x_s]\quad (1\le i\le r)
\]
be polynomials each of degree at most $t$, and with $\|P_i\|_1\le k$ $(1\le i\le r)$. 
Suppose that $\calA$ is a set of algebraic numbers with $d(\calA)\le (2t)^{2^A}$. Then 
$\calA$ is algebraic Freiman $\bfP$-isomorphic to a set of algebraic numbers $\calB$ with
\[
d(\calB)\le (2t)^{2^A}\quad \text{and}\quad \text{\rm Env}(\calB)\ll 
r^2A^{4s}(2kt)^{(2t)^{(2t)^{2^{A+6}}}}.
\]
Here, the implicit constant in Vinogradov's notation is absolute.
\end{theorem}

We remark that, by inflating the parameter $t$ so that $(2t)^{2^A}\ge d(\calA)$, the 
theorem can be applied so as to accomodate sets $\calA$ of algebraic integers of arbitrarily 
large finite degree $d(\calA)$. The following corollary may make the conclusion of Theorem 
\ref{theorem5.6} more transparent.

\begin{corollary}\label{corollary5.7}
In the setting of Theorem \ref{theorem5.6}, we have
\[
\text{\rm Env}_\del^*(\calA;\bfP)\le \exp_4(c_1A)\quad \text{with}\quad \del\le 
(2t)^{2^A},
\]
where $c_1=c_1(r,s,t,k)$ is a positive number depending at most on $r$, $s$, $t$ and $k$.
\end{corollary}

In certain situations, one may be interested in working with algebraic integers rather than 
more general algebraic numbers. For homogeneous polynomials, this is of course easily 
handled by clearing denominators.

\begin{corollary}\label{corollary5.8} Assume GRH for all Dedekind zeta functions. Let 
\[
P_i(\bfx)\in \dbZ[x_1,\ldots ,x_s]\quad (1\le i\le r)
\]
be homogeneous polynomials each of degree at most $t$, and with $\|P_i\|_1\le k$ 
$(1\le i\le r)$. Suppose that $\calA$ is a set of algebraic integers with 
$d(\calA)\le (2t)^{2^A}$. Then $\calA$ is algebraic Freiman $\bfP$-isomorphic to a set of 
algebraic integers $\calB$ with
\[
d(\calB)\le (2t)^{2^A}\quad \text{and}\quad \text{\rm Env}(\calB)\ll 
r^{2At}A^{4stA}(2kt)^{(2t)^{(2t)^{2^{A+7}}}}.
\]
Here, the implicit constant in Vinogradov's notation is absolute.
\end{corollary}

\begin{proof} Under the hypotheses of the statement of the corollary, it follows from 
Theorem \ref{theorem5.6} that $\calA$ is algebraic Freiman $\bfP$-isomorphic to a set of 
algebraic numbers $\calC$ with
\[
d(\calC)\le (2t)^{2^A}\quad \text{and}\quad \text{Env}(\calC)\ll 
r^2A^{4s}(2kt)^{\nu(A+6)}.
\]
Consider a typical element $c\in \calC$ and its minimal polynomial $m_c$ over $\dbZ$. For 
some integer $d=d_c$, we can write
\[
m_c(x)=g_0x^d+\ldots +g_{d-1}x+g_d,
\]
where $g_i=g_i(c)\in \dbZ$ satisfies $|g_i|\le \text{Env}(\calC)$ for $0\le i\le d$. Let $G$ to 
be the least common multiple of all of the integers $g_0(c)$ with $c\in \calC$, so that
\[
G\le \prod_{c\in \calC}g_0(c)\le (\text{Env}(\calC))^A
\]
and for each $c_0\in \calC$ one has
\[
G/g_0(c_0)\le \prod_{c\in \calC\setminus \{c_0\}}g_0(c)\le (\text{Env}(\calC))^{A-1}.
\]
Observe that when $c\in \calC$, we have the relation
\[
(G^d/g_0)m_c(x)=(Gx)^d+(G/g_0)g_1(Gx)^{d-1}+\ldots 
+(G^{d-1}/g_0)g_{d-1}(Gx)+(G^d/g_0)g_d,
\]
so that $Gc$ is an algebraic integer whose minimal polynomial $m_{Gc}$ satisfies
\[
\| m_{Gc}\|_1 \le (G^d/g_0(c))\text{Env}(\calC)\le (\text{Env}(\calC))^{dA}\le 
(\text{Env}(\calC))^{tA}.
\]

\par We consider the set
\[
\calB=\{ Gc:c\in \calC\}.
\]
It follows from the above discussion that $\calB$ is a set of algebraic integers with 
$d(\calB)=d(\calC)\le (2t)^{2^A}$ and
\[
\text{Env}(\calB)\le (\text{Env}(\calC))^{tA}\le r^{2At}A^{4Ast}(2kt)^{tA\nu(A+6)}.
\]
The conclusion of the corollary follows with a modicum of computation.
\end{proof}

\begin{corollary}\label{corollary5.9}
In the setting of Corollary \ref{corollary5.8}, the set of algebraic integers $\calA$ is 
algebraic Freiman $\bfP$-isomorphic to a set of algebraic integers $\calB$ with
\[
d(\calB)\le (2t)^{2^A}\quad \text{and}\quad \text{\rm Env}(\calB)\ll \exp_4(c_2A),
\]
where $c_2=c_2(r,s,t,k)$ is a positive number depending at most on $r$, $s$, $t$ and $k$.
\end{corollary}

We finish this section by remarking that, in certain non-linear situations, conclusions 
significantly stronger than are made available via Theorem \ref{theorem5.6} can be 
obtained by making use of underlying linear structure.

\begin{theorem}\label{theorem5.10}
Let $P_i(\bfx)\in \dbZ[x_1,\ldots ,x_s]$ $(1\le i\le r)$ be diagonal polynomials of the shape
\[
P_i(\bfx)=\sum_{j=1}^s c_{ij}x_j^t\quad (1\le i\le r),
\]
where $\|P_i\|_1\le k$ $(1\le i\le r)$. Suppose that $\calA$ is a finite set of integers. Then, 
when $\text{\rm card}(\calA)$ is large, one has
\[
\text{\rm Env}_\del^*(\calA;\bfP)\le t2^{t+1}(k+1)^A\quad \text{with}\quad 
\del\le t^A.
\]
\end{theorem}

\begin{proof} We consider the set of integers $\calA_t=\{a^t:a\in \calA\}$ and the set of 
linear polynomials
\[
L_i(\bfy)=\sum_{j=1}^s c_{ij}y_j\quad (1\le i\le r).
\]
By Theorem \ref{theorem3.2}, the set $\calA_t$ is Freiman ${\mathbf L}$-isomorphic to a 
set of integers $\calB_t$ with $\text{env}(\calB_t)\le (k+1)^A$. Now consider the set
\[
\calB=\{b^{1/t}:b\in \calB_t\}.
\]
One has $d(\calB)\le t^{\text{card}(\calB)}=t^A$. Moreover, given $c\in \calB$, one has 
$c^t\in \dbZ$ with $|c|^t<\text{env}(\calB_t)\le (k+1)^A$. By applying the corollary to the 
main theorem of Granville \cite{Gra1990}, much as in the conclusion of the proof of Lemma 
\ref{lemma5.5}, we find that the minimal polynomial $m_c$ of $c$ over $\dbZ$ is a divisor 
of the polynomial $f_{t,l}(x)=x^t-l$, where $l=c^t\in \calB_t$. Thus,
\begin{align*}
\|m_c\|_1&\le 2\text{deg}(f_{t,l})\Bigl( \frac{\sqrt{5}+1}{2}\Bigr)^{\text{deg}(f_{t,l})}
\|f_{t,l}\|_1\\
&\le t2^{t+1}\text{env}(\calB_t)\\
&\le t2^{t+1}(k+1)^A.
\end{align*}
Thus $\text{Env}(\calB)\le t2^{t+1}(k+1)^A$, and $\calB$ is algebraic Freiman 
$\bfP$-isomorphic to $\calA$ with $d(\calB)\le t^A$. This completes the proof of the 
theorem.
\end{proof}

\section{Densifications of sets} We turn next to a discussion of the densification idea to 
which we alluded in the introduction. We begin with an analogue of the Freiman 
$\bfP$-isomorphism defined in Definition \ref{definition2.1} suitable for the discussion of 
cartesian products. In this context, when $P_1,\ldots ,P_r\in \dbZ[x_1,\ldots ,x_s]$ and 
$\calC \subset \dbZ$, we again write
\[
S(\calC;\bfP)=\{ \bfx\in \calC^s:\text{$P_i(\bfx)=0$ $(1\le i\le r)$}\}.
\]
Also, when $\bfx_1,\ldots ,\bfx_t\in \calC^s$, we have in mind the notational convention 
that
\[
\bfx_i=(x_{i1},\ldots ,x_{is}).
\]
Then, when $(\bfx_1,\ldots ,\bfx_t)\in \calC^s\times \ldots \times \calC^s$, it is convenient 
to abbreviate the $t$-tuple $(x_{1j},x_{2j},\ldots ,x_{tj})$ as $\bfx^{(j)}$.

\begin{definition}\label{definition6.1} Let $t\in \dbN$, and suppose that $\calA$ and 
$\calB$ are finite sets of integers with $|\calB|=|\calA|^t$. Suppose in addition that 
$P_1,\ldots ,P_r\in \dbZ[x_1,\ldots ,x_s]$. We say that a bijection 
$\ome:\calA^t\rightarrow \calB$ is a $t$-fold Freiman $\bfP$-isomorphism (from $\calA$ to 
$\calB$) if it is the case that
\[
(\bfx_1,\ldots ,\bfx_t)\in S(\calA;\bfP)^t
\]
if and only if
\[
\left(\ome(\bfx^{(1)}),\ldots ,\ome(\bfx^{(s)})\right)\in S(\calB ;\bfP).
\]
\end{definition}

As in the discussion of \S2, we emphasise that a $t$-fold Freiman 
$\bfP$-isomorphism is specific to a particular polynomial tuple $\bfP$, and maps a $t$-tuple 
of integers to an integer. This once again permits an iterative approach in which $t$-fold 
Freiman $\bfP$-isomorphisms are successively composed in the natural manner.\par

It may be useful to highlight the utility of such a definition. When $t>1$, the structure of 
the solutions of the system of polynomials
\begin{equation}\label{6.1}
P_i(\bfx)=0\quad (1\le i\le r),
\end{equation}
with $\bfx\in \calA^s$, both determines and is determined by
\[
S(\calA;\bfP)^t=S(\calA;\bfP)\times \ldots \times S(\calA;\bfP).
\]
When $\ome:\calA^t\rightarrow \calB$ is a $t$-fold Freiman $\bfP$-isomorphism, it 
follows from Definition \ref{definition6.1} that $S(\calA;\bfP)^t$ is in bijective 
correspondence with $S(\ome(\calA^t);\bfP)=S(\calB;\bfP)$. Thus, the structure of the 
solutions of the system (\ref{6.1}) with $\bfx\in \calA^s$ both determines and is 
determined by the structure of the solutions of the system (\ref{6.1}) with 
$\bfx\in \calB^s$. A particularly simple consequence of this observation is that, just as 
$|\calB|=|\calA|^t$, so too one has
\[
|S(\calB;\bfP)|=|S(\calA;\bfP)|^t.
\]
Provided that $\text{env}(\calB)$ is not too much larger than $\text{env}(\calA)$, then the 
solution set $S(\calA;\bfP)$ of a sparse set $\calA$ may be understood precisely in terms 
of a potentially denser set $\calB$ and its solution set $S(\calB;\bfP)$. This motivates the 
next definition.

\begin{definition}\label{definition6.2} We say that a mapping 
$\ome:\calA^t\rightarrow \calD$ is a {\it $t$-fold $\bfP$-densifier of $\calA$} if it is a 
$t$-fold Freiman $\bfP$-isomorphism having the property that 
$\text{env}(\calD)\le \text{env}(\calA)^t$. When the latter inequality is strict, we refer to 
$\ome$ as a strict $t$-fold $\bfP$-densifier of $\calA$. In either case, we refer to $\calD$ 
as being a $t$-fold $\bfP$-densification of $\calA$.
\end{definition}

Of particular interest are the $t$-fold $\bfP$-densifications $\calD_t$ of $\calA$ 
distinguished by the property that
\[
\frac{\log \text{env}(\calD_t)}{\log |\calD_t|}
\]
is particularly small.

\begin{definition}\label{definition6.3} Let $\calA$ be a finite set of integers, and suppose 
that $P_1,\ldots ,P_r\in \dbZ[x_1,\ldots ,x_s]$. We say that the set $\calA$ has 
{\it $\bfP$-densification exponent $\kap$} when
\[
\kap =\liminf_{t\rightarrow \infty}\left\{ \frac{\log \text{env}(\calD_t)}{\log |\calD_t|} :
\text{$\calD_t$ is a $t$-fold densification of $\calA$}\right\} .
\]
\end{definition}

It follows that when $\calA$ has finite $\bfP$-densification exponent $\kap$, then for each 
$\eps>0$ there is a natural number $t$ and a $t$-fold $\bfP$-densification $\calD_t$ of 
$\calA$ such that
\[
\text{env}(\calD_t)\le |\calD_t|^{\kap+\eps}.
\]
Suppose that, in addition, one has an estimate of the shape
\[
|S(\calB;\bfP)|\ll \text{env}(\calB)^\eps |\calB|^\tet,
\]
valid for all finite sets of integers $\calB$. Then we may infer that
\[
|S(\calA;\bfP)|^t=|S(\calD_t;\bfP)|\ll \text{env}(\calD_t)^\eps |\calD_t|^\tet 
<|\calD_t|^{\tet+2\kap \eps}\ll |\calA|^{t(\tet+2\kap\eps)},
\]
whence
\[
|S(\calA;\bfP)|\ll |\calA|^{\tet+2\kap\eps}.
\]
In this way, it should be apparent that the existence of finite $\bfP$-densification exponents 
would lead from conclusions such as Theorem \ref{theorem1.1} to the validity of 
conjectures of the shape of that recorded in Conjecture \ref{conjecture1.2}. We shall see 
in the next section that, while such objectives are attainable for linear systems $\bfP$, it 
would seem that for systems $\bfP$ of higher degree, currently accessible conclusions are 
necessarily weaker.

\section{Densifications for linear systems of equations} The polynomial systems most 
amenable to densification via the circle of ideas already presented in \S3 are systems of 
homogeneous linear equations. Since the results concerning such systems are both simple 
and instructive, we expend the bulk of this section on their analysis. In order to fix ideas, 
suppose that $s\ge 2$, $r\ge 1$ and for $1\le i\le r$ one has $c_{ij}\in \dbZ$ 
$(1\le j\le s)$. We again ignore the trivial situation in which for some index $i$ one has 
$c_{ij}=0$ for $1\le j\le s$. The system of polynomials initially of interest to us in this 
section is
\[
P_i(\bfx)=\sum_{j=1}^sc_{ij}x_j\quad (1\le i\le r).
\]
Next, when $\calA\subset \dbZ$ is a finite set of integers, we recall the notation of writing 
$S(\calA;\bfP)$ for the set of solutions of the system of equations $P_i(\bfx)=0$ 
$(1\le i\le r)$, with $\bfx\in \calA^s$. In accordance with the treatment of \S3, we define 
$\Lam=\Lam(\bfc)$ by putting
\[
\Lam=\max_{1\le i\le r}\sum_{j=1}^s|c_{ij}|.
\]

\begin{theorem}\label{theorem7.1} Consider a system $\bfP$ of linear polynomials as 
described in the preamble, and consider a finite set of integers $\calA$. Then provided that 
$A=\text{\rm card}(\calA)$ is sufficiently large in terms of $r$ and $s$, the set $\calA$ has 
a finite $\bfP$-densification exponent $\kap$ satisfying $\kap\le s$. In particular, whenever 
$\eps>0$, there exists a natural number $t$ and a $t$-fold Freiman $\bfP$-isomorphism 
$\ome:\calA^t\rightarrow \calD$ having the property that 
$\text{\rm env}(\calD)\le \text{\rm card}(\calD)^{s+\eps}$.
\end{theorem}

\begin{proof} Fix a small positive number $\eps$. We seek to apply an iterative strategy 
that, given a set $\calD$ that is $t$-fold Freiman $\bfP$-isomorphic to $\calA$, generates a 
new set $\calD'$ that is $t'$-fold Freiman $\bfP$-isomorphic to $\calD$ and satisfies
\begin{equation}\label{7.1}
\frac{\log \text{env}(\calD')}{\log \text{card}(\calD')}\le (1-\eps)
\frac{\log \text{env}(\calD)}{\log \text{card}(\calD)}.
\end{equation}
Notice that the composition of a $t$-fold Freiman $\bfP$-isomorphism from $\calA$ to 
$\calD$, and a $t'$-fold Freiman $\bfP$-isomorphism from $\calD$ to $\calD'$, gives a 
$tt'$-fold Freiman $\bfP$-isomorphism from $\calA$ to $\calD'$. Thus, the relation 
(\ref{7.1}) suggests an improvement in the densification exponent. Provided that we are 
able to iterate this process sufficiently many times, we find that a $\bfP$-densification 
$\calD$ of $\calA$ exists with
\[
\frac{\log \text{env}(\calD)}{\log \text{card}(\calD)}\le (1-\eps)^n
\frac{\log \text{env}(\calA)}{\log \text{card}(\calA)},
\]
with $n$ as large as is necessary. It transpires that when $\text{env}(\calD)>
\text{card}(\calD)^{s+\eps}$, further iteration is possible, and in this way we see that the 
$\bfP$-densification exponent of $\calA$ is at most $s$.\par

We now initiate the proof proper. We may suppose without loss of generality that 
$\Lam\ge 2$ and $s\ge 2$. We consider a finite set of integers $\calD$ that is $t$-fold 
Freiman $\bfP$-isomorphic to $\calA$, so that $|\calD|=A^t\ge A$. Write $D=|\calD|$ and 
$X=\text{env}(\calD)-1$. If one were to have
\[
X+1\le D^{s(1+4\eps)},
\]
then the desired conclusion would follow, since $\eps>0$ may be taken arbitrarily small. 
We may therefore suppose that $X+1>D^{s(1+4\eps)}$.\par

Next, in accordance with (\ref{3.1}), we define the natural number $\Ups$ by putting
\[
\Ups=\Biggl( \prod_{\substack{d_1,d_2\in \calD \\ d_1\ne d_2}}|d_i-d_j|\Biggr) \Biggl( 
\prod_{\substack{\bfd\in \calD^s\\ \bfd\not\in S(\calD;\bfP)}}\sum_{i=1}^r|P_i(\bfd)|
\Biggr) .
\]
Then one finds that
\[
1\le \Ups\le (2X)^{D^2}(r\Lam X)^{D^s}\le \tfrac{1}{3}(r\Lam X)^{2D^s}.
\]
Note that
\[
2\log (3\Ups)\le 4D^s\log (r\Lam X).
\]
Then provided that $Y\ge 4D^s\log (r\Lam X)$, it follows from the prime number theorem 
that in any interval $(Y,2Y)$, there exist at least $D$ prime numbers $\pi$ with 
$\pi\nmid \Ups$. Let $\pi_1,\ldots ,\pi_D$ be any $D$ such distinct prime numbers.\par

We next construct a map $\ome:\calD^D\rightarrow \dbZ$ as follows. When 
$\bfd=(d_1,\ldots ,d_D)\in \calD^D$, we define
\begin{equation}\label{7.2}
\ome(\bfd)=\sum_{i=1}^Dd_i\prod_{\substack{1\le j\le D\\ j\ne i}}\pi_j.
\end{equation}
Write $\calE=\ome(\calD^D)$. Then we claim that the mapping 
$\ome:\calD^D\rightarrow \calE$ is a $D$-fold Freiman $\bfP$-isomorphism from $\calD$ to 
$\calE$.\par

We first verify that $\ome:\calD^D\rightarrow \calE$ is a bijection, and for this it suffices 
to check that $\ome$ is injective. However, if $\bfd,\bfd'\in \calD^D$ and 
$\ome(\bfd)=\ome(\bfd')$, then it is apparent that
\[
\sum_{i=1}^Dd_i\prod_{\substack{1\le j\le D\\ j\ne i}}\pi_j\equiv \sum_{i=1}^Dd'_i
\prod_{\substack{1\le j\le D\\ j\ne i}}\pi_j\mmod{\pi_k}\quad (1\le k\le D),
\]
whence
\[
(d_k-d'_k)\prod_{\substack{1\le j\le D\\ j\ne k}}\pi_j\equiv 0\mmod{\pi_k}\quad 
(1\le k\le D).
\]
For each index $k$, however, one has
\[
\Biggl(\pi_k, \prod_{\substack{1\le j\le D\\ j\ne k}}\pi_j\Biggr) =1,
\]
and thus we deduce that $d_k\equiv d'_k\mmod{\pi_k}$ $(1\le k\le D)$. Recalling the 
definition of $\Ups$, however, one sees that $\pi_k\nmid (d_k-d'_k)$ whenever 
$d_k\ne d'_k$, and so we must have $d_k=d'_k$ $(1\le k\le D)$. In this way, we conclude 
that $\bfd=\bfd'$, whence $\ome: \calD^D\rightarrow \calE$ is indeed bijective.\par

Next, whenever $(\bfd_1,\ldots ,\bfd_D)\in S(\calD;\bfP)^D$, the linearity of the 
polynomials $\bfP$ ensures that for $1\le l\le r$, one has
\[
P_l\left(\ome(\bfd^{(1)}),\ldots ,\ome(\bfd^{(s)})\right) =\sum_{i=1}^D
P_l(d_{i1},\ldots ,d_{is})\prod_{\substack{1\le j\le D\\ j\ne i}}\pi_j=0.
\]
Thus $\left(\ome(\bfd^{(1)}),\ldots ,\ome(\bfd^{(s)})\right) \in S(\calE;\bfP)$. Also, when 
$(\bfd_1,\ldots ,\bfd_D)\not\in S(\calD;\bfP)^D$, then for some index $l$ with 
$1\le l\le r$, and some index $k$ with $1\le k\le D$, one has
\[
P_l(d_{k1},\ldots ,d_{ks})\ne 0.
\]
Meanwhile, if one were to have
\begin{equation}\label{7.3}
P_l\left(\ome(\bfd^{(1)}),\ldots ,\ome(\bfd^{(s)})\right)=0\quad (1\le l\le r),
\end{equation}
then in particular,
\[
\sum_{i=1}^DP_l(d_{i1},\ldots ,d_{is})\prod_{\substack{1\le j\le D\\ j\ne i}}\pi_j
\equiv 0\mmod{\pi_k}\quad (1\le l\le r).
\]
The latter congruences imply that
\[
P_l(d_{k1},\ldots ,d_{ks})\prod_{\substack{1\le j\le D\\ j\ne k}}\pi_j\equiv 0
\mmod{\pi_k}\quad (1\le l\le r),
\]
whence
\begin{equation}\label{7.4}
P_l(d_{k1},\ldots ,d_{ks})\equiv 0\mmod{\pi_k}\quad (1\le l\le r).
\end{equation}
But the definition of $\Ups$ ensures that when $\bfd_k\not \in S(\calD;\bfP)$, as we may 
assume, then
\[
\sum_{l=1}^r|P_l(\bfd_k)|\not\equiv 0\mmod{\pi_k}.
\]
Thus we have $P_l(\bfd_k)\not\equiv 0\mmod{\pi_k}$ for some index $l$ with 
$1\le l\le r$, and this contradicts the relation (\ref{7.4}). We therefore conclude that 
(\ref{7.3}) cannot hold. In consequence, when 
$(\bfd_1,\ldots ,\bfd_D)\not\in S(\calD;\bfP)^D$, one must have
\[
\left( \ome(\bfd^{(1)}),\ldots ,\ome(\bfd^{(s)})\right)\not\in S(\calE;\bfP).
\]
We have thus shown that $\ome:\calD^D\rightarrow \calE$ is a $D$-fold Freiman 
$\bfP$-isomorphism.\par

We next investigate the $\bfP$-densification exponent associated with the mapping 
$\ome:\calD^D\rightarrow \calE$. Observe first that the definition (\ref{7.2}) shows that
\[
\text{env}(\calE)\le D(2Y)^{D-1}\max_{1\le i\le D}|d_i|\le D(2Y)^{D-1}\text{env}(\calD).
\]
We take $Y=4D^s\log (r\Lam X)$, in which we recall that $X=\text{env}(\calD)-1$. Thus
\begin{align*}
\frac{\log \text{env}(\calE)}{\log |\calE|}&\le 
\frac{\log \text{env}(\calD)+\log D+(D-1)\log (2Y)}{D\log D}\\
&=\frac{1}{D}\left( \frac{\log \text{env}(\calD)}{\log D}\right) +\left( 1-\frac{1}{D}
\right) \left( \frac{\log (2Y)}{\log D}\right) +\frac{1}{D}.
\end{align*}
It follows that whenever
\begin{equation}\label{7.5}
\frac{\log (2Y)}{\log D}\le (1-2\eps)\frac{\log \text{env}(\calD)}{\log D},
\end{equation}
then one has
\begin{equation}\label{7.6}
\frac{\log \text{env}(\calE)}{\log |\calE|}\le 
(1-\eps)\frac{\log \text{env}(\calD)}{\log |\calD|}.
\end{equation}
This is the improving $\bfP$-densification argument outlined in the opening discussion of 
the proof.\par

Let us return to examine the condition (\ref{7.5}). This condition is satisfied provided that
\[
2Y\le (\text{env}(\calD))^{1-2\eps}=(X+1)^{1-2\eps},
\]
which is to say that
\[
8D^s\log (r\Lam X)\le (X+1)^{1-2\eps}.
\]
However, in the opening discussion of the proof, we were at liberty to suppose that 
$X+1>D^{s(1+4\eps)}$. Thus we have
 \[
\frac{(X+1)^{1-2\eps}}{\log (r\Lam X)}>D^{s(1+\eps)}>8D^s,
\]
and in consequence the condition (\ref{7.5}) is fulfilled. This justifies the conclusion 
(\ref{7.6}).\par

As we explained in the opening discussion of the proof, the upper bound (\ref{7.6}) 
permits an iterative approach to be employed that delivers a $t$-fold $\bfP$-densification 
$\calD$ of $\calA$ satisfying the property that
\begin{equation}\label{7.7}
\frac{\log \text{env}(\calD)}{\log \text{card}(\calD)}\le (1-\eps)^n
\frac{\log \text{env}(\calA)}{\log \text{card}(\calA)},
\end{equation}
with $n$ arbitarily large, provided only that 
$\text{env}(\calD)>\left( \text{card}(\calD)\right)^{s(1+4\eps)}$. Since for sufficiently 
large $n$, the bound (\ref{7.7}) contradicts the condition 
$\text{env}(\calD)>\left( \text{card}(\calD)\right)^{s(1+4\eps)}$, we are forced to 
conclude that such a $t$-fold $\bfP$-densification $\calD$ exists in which 
$\text{env}(\calD)\le \left( \text{card}(\calD)\right)^{s(1+4\eps)}$. By taking $\eps>0$ 
arbitrarily small, this shows that
\[
\liminf_{t\rightarrow \infty}\left\{ \frac{\log \text{env}(\calD_t)}{\log |\calD_t|}:
\text{$\calD_t$ is a $t$-fold $\bfP$-densification of $\calA$}\right\}\le s.
\]
This completes the proof of the theorem.
\end{proof}

The strategy underlying the proof of Theorem \ref{theorem7.1} can be generalised in 
some sense both to inhomogeneous systems, and also to systems of equations of degree 
exceeding $1$. In order to illustrate ideas, consider a system of homogeneous polynomials 
$P_i(\bfx)\in \dbZ[x_1,\ldots ,x_s]^r$, not necessarily linear. Suppose that these 
polynomials are of degree at most $k$, and that the sum of the absolute values of the 
coefficients in the polynomial $P_i(\bfx)$ is at most $\Lam$ for $1\le i\le r$. Let $\calD$ 
be a finite set of integers that is $t$-fold Freiman $\bfP$-isomorphic to $\calA$, and write 
$D=|\calD|$ and $X=\text{env}(\calD)$. Also, define the integer $\Ups$ now by putting
\[
\Ups=\Biggl( \prod_{\substack{d_1,d_2\in \calD \\ d_1\ne d_2}}|d_1-d_2|\Biggr) 
\Biggl( \prod_{i=1}^r\prod_{\substack{\bfd\in \calD^s\\ P_i(\bfd)\ne 0}}|P_i(\bfd)|\Biggr) .
\]
Then provided that $Y\ge 4rD^{2s}\log (\Lam X^k)$, it follows from the prime number 
theorem that in any interval $(Y,2Y)$, there exist at least $D$ prime numbers $\pi$ with 
$\pi\nmid \Ups$. Let $\pi_1,\ldots ,\pi_D$ be any $D$ such distinct prime numbers.\par

We again define a map $\ome:\calD^D\rightarrow \dbZ/(\pi_1\ldots \pi_D\dbZ)$ via 
(\ref{7.2}), and write $\calE=\ome (\calD^D)$. The map $\ome$ is a bijection from 
$\calD^D$ to $\calE$, just as in the analogous argument in the proof of Theorem 
\ref{theorem7.1}. We observe that for $1\le l\le r$, one has
\[
P_l\left(\ome(\bfd^{(1)}),\ldots ,\ome(\bfd^{(s)})\right) \equiv 
\sum_{i=1}^D P_l(d_{i1},\ldots ,d_{is})\Biggl( \prod_{\substack{1\le j\le D\\ j\ne i}}
\pi_j\Biggr)^{\text{deg}(P_l)}\mmod{\pi_1\ldots \pi_D}.
\]
If $(\bfd_1,\ldots ,\bfd_D)\not\in S(\calD;\bfP)^D$, then for some index $l$ with 
$1\le l\le r$, and some index $k$ with $1\le k\le D$, one has
\[
P_l(d_{k1},\ldots ,d_{ks})\ne 0.
\]
Since $\pi_j\nmid \Ups$ for $1\le j\le D$, one cannot have
\[
P_l(d_{k1},\ldots ,d_{ks})\equiv 0\mmod{\pi_k},
\]
and consequently
\[
P_l\left(\ome(\bfd^{(1)}),\ldots ,\ome (\bfd^{(s)})\right)\not\equiv 0
\mmod{\pi_1\ldots \pi_D}\quad (1\le l\le r).
\]
On the other hand, whenever $(\bfd_1,\ldots ,\bfd_D)\in S(\calD;\bfP)^D$, then for 
$1\le l\le r$ one must have
\[
P_l\left(\ome(\bfd^{(1)}),\ldots ,\ome(\bfd^{(s)})\right)\equiv 0\mmod{\pi_1\cdots \pi_D}.
\]
We thus perceive that the solution structure of $S(\calD;\bfP)^D$ is preserved by the map 
$\ome$ in a manner analogous to that in our discussion of densifications. One can now 
attempt to rectify the set $\ome(\calD^D)\subseteq \dbZ/(\pi_1\ldots \pi_D\dbZ)$ to obtain 
a new set $\calF\subset {\overline \dbQ}$ by means of the method of Grosu 
\cite{Gro2014}. In this way one perceives the possibility of a densification process for sets 
of algebraic numbers. However, in common with the method of Grosu, there is only weak 
control of the degree and other data associated with the field extension in which the 
elements of $\calF$ are embedded. This level of control would appear to be far too weak to 
facilitate useful densification conclusions.\par

We finish this section with some comments concerning the main conclusion of Theorem 
\ref{theorem7.1}. We are interested in understanding the set of solutions $S(\calA;\bfP)$ 
of a given system of polynomial equations $P_i(\bfx)=0\quad (1\le i\le r)$, with variables 
restricted to a set $\calA$. The conclusion of Theorem \ref{theorem7.1} shows that, in 
circumstances wherein the polynomials $P_i(\bfx)$ are both homogeneous and linear at 
least, this objective can be achieved by studying instead a related set of integers $\calD$ 
with $\text{env}(\calD)\le |\calD|^{s+\eps}$. While this polynomial dependence of 
$\text{env}(\calD)$ on $|\calD|$ may seem significantly superior to the exponential 
dependence available in the condensation results of \S3, one may interpret this 
nonetheless as a ``non-result''. If it is the case that $\calD$ is a typical set having roughly 
$X^{1/s-\eps}$ elements in a box of size $X$, then conventional heuristics suggest 
nothing more than that the number of solutions of the system $P_i(\bfx)=0$ 
$(1\le i\le r)$, with $\bfx\in \calD^s$, could be $O(1)$ or even $0$. In other words, the 
exponent $1/s$ is already small enough that in general little or nothing can be learned from 
the counting function for $|\calD\cap [1,X]|$ alone. Perhaps it is more illuminating to point 
out that more or less any solution behaviour can be encoded in a set $\calD$ for which 
$|\calD\cap [1,X]|\ll X^{1/s-\eps}$.

\section{Remarks on sets of real points} We now explore some consequences of work of 
Vu, Wood and Wood \cite[Theorem 1.1]{VWW2011}. Let $D$ be an integral domain of 
characteristic zero, such as the field of real numbers $\dbR$, and let $\calD$ be a finite 
subset of $D$. Consider a system of polynomials $P_i(\bfx)\in \dbZ[x_1,\ldots ,x_s]$ 
$(1\le i\le r)$. In this section, we are interested in the set $S(\calD;\bfP)$ of solutions 
$\bfx\in \calD^s$ of the simultaneous equations
\[
P_i(x_1,\ldots ,x_s)=0\quad (1\le i\le r).
\]
The structure of the solution set $S(\calD;\bfP)$ is determined by the hypergraph 
$\Gam(\calD;\bfP)$ defined just as in the analogous discussion of \S2.\par

Given a large prime number $p$, one may seek a ring homomorphism 
$\varphi_p:\dbZ[\calD]\rightarrow \dbF_p$ with the property that, whenever 
$(x_1,\ldots ,x_s)\in \calD^s$, then
\[
P_i(x_1,\ldots ,x_s)=0\quad (1\le i\le r)
\]
if and only if
\[
P_i(\varphi_p(x_1),\ldots ,\varphi_p(x_s))=0\quad (1\le i\le r).
\]
We emphasise here that the latter system of equations over $\dbF_p$ amount to a 
system of congruences. The conclusion of \cite[Theorem 1.1]{VWW2011} demonstrates 
that there exists an infinite sequence of primes with positive relative density having the 
property that such a ring homomorphism exists. This conclusion may not at first sight be 
obvious from \cite[Theorem 1.1]{VWW2011}. Of course, any ring homomorphism 
$\varphi_p:\dbZ[\calD]\rightarrow \dbF_p$ has the property that, whenever 
$(x_1,\ldots ,x_s)\in \calD^s$ satisfies $P_i(x_1,\ldots ,x_s)=0$ $(1\le i\le r)$, then 
\begin{equation}\label{8.1}
P_i(\varphi_p(x_1),\ldots ,\varphi_p(x_s))=\varphi_p(P_i(x_1,\ldots ,x_s))=\varphi_p(0)=0
\quad (1\le i\le r).
\end{equation}
Thus, the interesting feature for us is that whenever
\[
(x_1,\ldots ,x_s)\in \calD^s\quad \text{and}\quad P_i(x_1,\ldots ,x_s)\ne 0
\]
for some index $i$ with $1\le i\le r$, then
\[
P_i(\varphi_p(x_1),\ldots ,\varphi_p(x_s))=\varphi_p(P_i(x_1,\ldots ,x_s))\ne 0.
\]
The approach here is to define a set $L$ of all elements
\[
P_i(x_1,\ldots ,x_s)\in \dbZ[\calD],
\]
with $\bfx\in \calD^s$, having the property that $P(x_1,\ldots ,x_s)\ne 0$. The conclusion 
of \cite[Theorem 1.1]{VWW2011} guarantees that the ring homomorphisms $\varphi_p$, 
whose existence is asserted, may be constructed in such a manner that 
$0\not\in \varphi_p(L)$. This last assertion guarantees that the condition (\ref{8.1}) holds, 
and this ensures that the sought after ring homomorphisms $\varphi_p$ do indeed exist.
\par

Equipped with these ring homomorphisms $\varphi_p:\dbZ[\calD]\rightarrow \dbF_p$, we 
see that $\Gam (\calD;\bfP)$ is isomorphic as a hypergraph to 
$\Gam(\varphi_p(\calD);\bfP)$. Thus, the solution structure of $S(\calD;\bfP)$ may be 
faithfully embedded into appropriate finite fields $\dbF_p$. If the prime number $p$ has 
been chosen sufficiently large, then one may apply \cite[Theorem 1.3]{Gro2014} to 
obtain a faithful model of the finite field solution structure inside a number field $K$ 
with degree at most $\exp(\exp(c_\bfP |\calD|))$, for a suitable real number $c_\bfP$ 
depending at most on $\bfP$. For systems of linear equations, moreover, one can restrict 
to an integer model. In this way, one sees that linear problems involving sets of real 
points, for example, may be considered instead as linear problems involving sets of 
integers. For non-linear polynomial problems, we must instead work with sets of algebraic 
numbers of bounded algebraic enveloping radius. In both settings, the condensation and 
densification ideas of this paper become applicable.

\section{The conclusion of Theorem \ref{theorem1.1}}
As promised in the introduction, we briefly justify the conclusion of Theorem 
\ref{theorem1.1}. Suppose that $\calA\subset \dbZ$ is finite with $A=\text{card}(\calA)$, 
and define
\[
\gra_n=\begin{cases}1,&\text{when $n\in \calA$},\\
0,&\text{when $n\not\in \calA$}.\end{cases}
\]
Suppose first that $\varphi_j\in \dbZ[t]$ $(1\le j\le k)$ is a system of polynomials with 
\[
\text{\rm det}\biggl( \frac{{\rm d}^i\varphi_j(t)}{{\rm d}t^i}\biggr)_{1\le i,j\le k}\ne 0.
\]
Let $s$ and $k$ be natural numbers with $s\le k(k+1)/2$. Then for each $\eps>0$, the 
conclusion of \cite[Theorem 1.1]{Woo2019} shows that
\begin{align*}
\int_{[0,1)^k}\biggl| \sum_{|n|\le X}\gra_ne(\alp_1\varphi_1(n)+\ldots 
+\alp_k\varphi_k(n))\biggl|^{2s}\d\bfalp
&\ll X^\eps \biggl( \sum_{|n|\le X}|\gra_n|^2\biggr)^s\\
&\ll X^\eps A^s.
\end{align*}
Since for each $n\in \calA$, one has $|n|\le \text{env}(\calA)$, the first conclusion of 
Theorem \ref{theorem1.1} follows on setting $X=\text{env}(\calA)$.\par

The second conclusion of Theorem \ref{theorem1.1} follows on making use of the 
translation invariance property of the system of equations
\[
x_1^j+\ldots +x_s^j=x_{s+1}^j+\ldots +x_{2s}^j\quad (1\le j\le k).
\]
Put $m=\min \calA$, and observe that whenever $\bfx\in \calA^{2s}$ satisfies this system 
of equations, then as a consequence of the binomial theorem, one has
\[
(x_1-m)^j+\ldots +(x_s-m)^j=(x_{s+1}-m)^j+\ldots +(x_{2s}-m)^j\quad (1\le j\le k).
\]
Thus, if we put $\calB=\{a-m:a\in \calA\}$, then we have $J_{s,k}(\calA)=J_{s,k}(\calB)$. 
We therefore deduce from the special case $\varphi_j(t)=t^j$ $(1\le j\le k)$ of the first 
part of the theorem that
\[
J_{s,k}(\calA)\le (\text{env}(\calB))^\eps A^s=
(\max(\calA)-\min(\calA)+1)^\eps A^s=(\text{diam}(\calA))^\eps A^s.
\]
The second conclusion of Theorem \ref{theorem1.1} follows when $s\le k(k+1)/2$. When 
instead $s>k(k+1)/2$, we observe that a trivial estimate combines with orthogonality to 
show that
\begin{align*}
J_{s,k}(\calA)&\le A^{2s-k(k+1)}\int_{[0,1)^k}
\biggl| \sum_{|n|\le X}\gra_ne(\alp_1n+\ldots +\alp_kn^k)\biggl|^{k(k+1)}\d\bfalp \\
&\ll A^{2s-k(k+1)}\cdot (\text{diam}(\calA))^\eps A^{k(k+1)/2}.
\end{align*}
The desired conclusion is now immediate in this case, since $A^{2s-k(k+1)/2}>A^s$.

\bibliographystyle{amsbracket}
\providecommand{\bysame}{\leavevmode\hbox to3em{\hrulefill}\thinspace}

\end{document}